\documentclass[11pt,a4paper]{amsart}

\usepackage{amssymb}
\usepackage[utf8]{inputenc}
\usepackage[english]{babel}
\usepackage{amsfonts,mathtools}
\usepackage{amsmath}
\usepackage{comment}
\usepackage{transparent}
\usepackage{graphicx}
\usepackage{xcolor}
\usepackage{soul}
\usepackage{tikz}
\usetikzlibrary{calc}
\setstcolor{teal}
\usepackage{todonotes}

\DeclareFontFamily{U}{mathx}{\hyphenchar\font45}
\DeclareFontShape{U}{mathx}{m}{n}{
      <5> <6> <7> <8> <9> <10>
      <10.95> <12> <14.4> <17.28> <20.74> <24.88>
      mathx10
      }{}
\DeclareSymbolFont{mathx}{U}{mathx}{m}{n}
\DeclareFontSubstitution{U}{mathx}{m}{n}
\DeclareMathAccent{\widecheck}{0}{mathx}{"71}
\DeclareMathAccent{\wideparen}{0}{mathx}{"75}

\usepackage[bookmarksopen,bookmarksdepth=3,colorlinks,citecolor=red,pagebackref,hypertexnames=false]{hyperref}
\usepackage[nobysame,non-sorted-cites, initials]{amsrefs}

\usepackage{amsthm}
\usepackage{color}
\usepackage{enumitem}
\usepackage[nameinlink]{cleveref}

\newtheorem{proposition}{Proposition}[section]
\newtheorem{theorem}[proposition]{Theorem}

\newtheorem{corollary}[proposition]{Corollary}
\newtheorem{lemma}[proposition]{Lemma}
\theoremstyle{remark}
\newtheorem{remark}[proposition]{Remark}

\theoremstyle{definition}

\newtheorem*{acknowledgements}{Acknowledgements}

\DeclareMathOperator{\supp}{supp}

\newcommand{\R}{\mathbb{R}}
\newcommand{\C}{\mathbb{C}}
\newcommand{\Z}{\mathbb{Z}}

\newcommand{\Sph}{\mathbb{S}}
\newcommand{\dd}{\mathrm{d}}

\newcommand{\Id}{\mathrm{Id}}

\newcommand{\X}{\mathrm{X}}

\def\ve{\varepsilon}
\def\lec{\lesssim}
\def\vp{\varphi}
\def\wh{\widehat}
\def\ind{\mathbf 1}
\def\cal{\mathcal}
\def\ovl{\overline}
\def\cd{\centerdot}

\def\d{\partial}

\title[Initial-to-final-state inverse problem]{The initial-to-final-state inverse problem with critically-singular potentials}

\author[M. Cañizares]{Manuel Cañizares}
\address[M.\ Cañizares]{Johann Radon Institute for Computational and Applied Mathematics (RICAM)\\ Altenbergstr. 69, 4040 Linz, Austria.}
\email{\href{mailto:manuel.canizares@ricam.oeaw.ac.at}{\textrm{manuel.canizares@ricam.oeaw.ac.at}}}

\author[P. Caro]{Pedro Caro}
\address[P. Caro]{
Basque Center for Applied Mathematics and Ikerbasque (Basque Foundation for Science) Bilbao, Spain}
\email{\href{mailto:pcaro@bcamath.org}{\textrm{pcaro@bcamath.org}}}

\author[I. Parissis]{Ioannis Parissis}
\address[I.\ Parissis]{Departamento de Matem\'aticas, Universidad del Pa\'is Vasco, Aptdo. 644, 48080 Bilbao, Spain and Ikerbasque, Basque Foundation for Science, Bilbao, Spain}
\email{\href{mailto:ioannis.parissis@ehu.eus}{\textrm{ioannis.parissis@ehu.eus}}}

\author[T. Zacharopoulos]{Thanasis Zacharopoulos}
\address[T.\ Zacharopoulos]{Department of Mathematics, Aarhus University, NY Munkegade 118, 8000 Aarhus C, Denmark}
\email{\href{mailto:thanzacharop@math.au.dk}{\textrm{thanzacharop@math.au.dk}}}

\begin{document}

\begin{abstract} The Schr\"odinger equation in high dimensions describes the evolution of a quantum system. Assume that we are given the evolution map sending each initial state $f\in L^2(\R^n)$ of the system to the corresponding final state at a fixed time $T$. The main question we address in this paper is whether this \emph{initial-to-final-state} map uniquely determines the Hamiltonian $-\Delta+V$ that generates the evolution. We restrict attention to \emph{time-independent} potentials $V$ and show that uniqueness holds provided $V \in L^1(\R^n)\cap L^q(\R^n)$, with $q>1$ if $n=2$ or $q\geq n/2$ if $n\geq 3$. This should be compared with the results of Caro and Ruiz, who proved that in the time-dependent case, uniqueness holds under the stronger assumption that the potential exhibits super-exponential decay at infinity, for both bounded and unbounded potentials.

This paper extends earlier work of the same authors, where uniqueness was obtained for bounded time-independent potentials with polynomial decay at infinity. Here we only require $L^1$-type decay at infinity and allow for $L^q$-type singularities. We reach this improvement by providing a refinement of the Kenig--Ruiz--Sogge resolvent estimate, which replaces the classical Agmon--H\"ormander estimates used previously. Crucially, the time-independent setting allows us to avoid the use of complex geometrical optics solutions and thereby dispense with strong decay assumptions at infinity.
\end{abstract}

\date{\today}


\maketitle

\section{Introduction}
In this paper, we study an inverse problem for the Schr\"odinger equation in which the available data consist of the map that sends any initial state $f$ at time $t=0$ to the solution at a fixed final time $t=T$.

To make this precise, let us consider the initial-value problem for the Schr\"odinger equation with a time-independent potential
\begin{equation}\label{eq:Schrodinger}
\begin{cases}
i \partial_t u = -\Delta u + V u & \text{for } (t,x)\in (0, T)\times \R^n \eqqcolon \Sigma,
\vspace{.6em}
\\
u(0,x) = f(x) & \text{for } x\in \R^n.
\end{cases}
\end{equation}
We assume throughout the paper that $V=V(x)\in L^q(\R^n)$, where $q\ge n/2$ if $n\ge 3$ or $q>1$ if $n=2$. Then, by \cites{zbMATH02204588,zbMATH00179225}, this direct problem is well posed, and for every $f\in L^2(\R^n)$ there exists a unique solution
\[
u\in C\left([0,T];L^2(\R^n)\right).
\]
The evolution associated with \eqref{eq:Schrodinger} therefore defines a bounded linear operator
\[
\mathcal U : f\in L^2(\R^n)\mapsto u\in C\left([0,T];L^2(\R^n)\right).
\]
Consequently, for any fixed time $t\in[0,T]$, the operator
\[
\mathcal U_t : f \in L^2(\R^n) \mapsto u(t,\cd)\in L^2(\R^n)
\]
is also bounded, uniformly in $t$. Solutions of the form $u=\mathcal U f$, with $f\in L^2(\R^n)$, will be referred to as \emph{physical solutions}, while we call $\mathcal U_T$ the \emph{initial-to-final-state map}.

The main question addressed in this paper is whether the initial-to-final-state map $\mathcal U_T$ uniquely determines the Hamiltonian $-\Delta+V$. This inverse problem was first studied for \emph{time-dependent} potentials in \cite{zbMATH07801151}. There, the authors show that if the potentials $V_1,V_2 \in L^1((0,T);L^\infty(\R^n))$ satisfy a \emph{super-exponential decay} condition at infinity, and if $\mathcal U_T^j$ denotes the initial-to-final-state map associated with $-\Delta+V_j$, then there holds:
\[
\mathcal U_T^1=\mathcal U_T^2 \quad\Longrightarrow\quad V_1=V_2.
\]
More recently, this uniqueness result was extended in \cite{caro2025initialtofinalstateinverseproblemunbounded} to unbounded time-dependent potentials that are allowed to exhibit local $L^q$-type singularities, but still requiring the super-exponential decay assumption at infinity.

The case of time-independent potentials was previously considered in \cite{zbMATH08122191}, where uniqueness was established under comparatively weaker decay assumptions than in the time-dependent setting, namely assuming only super-linear decay at infinity. The purpose of the present paper is to relax these assumptions even further. Specifically, we prove uniqueness for time-independent potentials that may exhibit singularities of $L^q$-type in sets of finite measure and satisfy only $L^1$-integrability in sets of infinite measure; in the time-independent setting, this represents a substantial improvement over the decay and integrability assumptions in  \cite{zbMATH08122191}, \cite{zbMATH07801151}, and \cite{caro2025initialtofinalstateinverseproblemunbounded}.

\begin{theorem}\sl \label{th:Lq} Let $V_1,V_2\in L^1(\R^n)\cap L^q(\R^n)$ be time-independent potentials, where $q>1$ if $n=2$ and $q\ge n/2$ if $n\ge 3$. Let $\mathcal U_T^1$ and $\mathcal U_T^2$ denote the corresponding initial-to-final-state maps. Then
\[
\mathcal U_T^1=\mathcal U_T^2 \quad\Longrightarrow\quad V_1=V_2.
\]
\end{theorem}

\medskip
\noindent\textbf{Outline of the proof.}
The proof of Theorem~\ref{th:Lq} proceeds by extracting information on the difference $V_1-V_2$ from $\mathcal U_T^1 = \mathcal U_T^2$, by testing it against suitable families of solutions.

The first step is to show that the equality $\mathcal U_T^1=\mathcal U_T^2$ yields an \emph{Alessandrini-type orthogonality relation}, in the sense of \cite{Aless}, of the form
\begin{equation}\label{eq:orthogonality_intro}
\int_{\Sigma} (V_1-V_2)\,u_1\,\overline{v_2}=0,
\end{equation}
valid for pairs of solutions associated with the potentials $V_1$ and $\ovl{V_2}$. This identity is initially available only for \emph{physical solutions}, namely solutions belonging to $C([0,T];L^2(\R^n))$. 

At this stage, our first obstruction is that in order to recover pointwise information on $V_1-V_2$, we need to test \eqref{eq:orthogonality_intro} against special \emph{time-harmonic} solutions. Such solutions can be constructed as perturbations of time-harmonic solutions of the  free Schr\"odinger equation; we refer to such solutions as \emph{stationary states} associated with the given potential. However, these stationary states are not in $C([0, T]; L^2(\R^n))$, and using them as test functions requires a substantial extension of the Alessandrini-type orthogonality relation beyond physical solutions.

More precisely, stationary states are obtained by inverting a resolvent-type operator and constructing correction terms via Neumann series. Concretely, they are of the form
\[
u(t,x)=e^{-i|\kappa|^2 t}\left(e^{-i\kappa\cdot x}+w^{\mathrm{cor}}(x)\right),
\]
where the leading term \(e^{-i|\kappa|^2 t}e^{-i\kappa\cdot x}\) is a time-harmonic solution of the free Schr\"odinger equation, and the correction \(w^{\mathrm{cor}}\) accounts for the presence of the potential \(V\).

The construction of \(w^{\mathrm{cor}}\) reduces to inverting an operator of the form \(\mathrm{Id}-P_\lambda\circ V\), where \(P_\lambda\) denotes a solution operator for the Helmholtz equation \((\Delta+\lambda^2)u=f\). This inversion is carried out on suitable function spaces, chosen so that the operator \(P_\lambda\circ V\) is small, in a suitable sense.

In following the proof strategy outlined above, we encounter a second important obstruction. In the non-endpoint regime $V\in L^q$ with $q>n/2$, decay in the energy parameter $\lambda$ follows from the classical Kenig--Ruiz--Sogge resolvent estimate. The latter ensures that $P_\lambda\circ V$ is small on the relevant function spaces for large $\lambda$. At the critical endpoint $q=n/2$, however, this estimate no longer yields decay in $\lambda$, and the smallness of $P_\lambda\circ V$ cannot be obtained from the standard resolvent bound alone.

To address this issue, we introduce a new scale of Banach spaces that allows us to recover a decaying factor in the energy parameter, also at the endpoint. Within this framework, we establish an improved form of the Kenig--Ruiz--Sogge resolvent estimate, sufficient to construct stationary states for critical potentials despite the absence of quantitative decay at the endpoint. A similar idea has appeared in \cite{CaroGarcia} for a related scattering problems with critically singular potentials; see also \cite{zbMATH07867333}. When inserted into the extended Alessandrini-type orthogonality relation, these stationary states allow us to isolate the Fourier phase and hence recover the Fourier transform of $V_1-V_2$, proving uniqueness.

Throughout the argument, the assumption of a stationary potential plays a crucial role. Indeed, in the time-dependent setting, the uniqueness results in \cites{zbMATH07801151,caro2025initialtofinalstateinverseproblemunbounded} rely on constructing complex geometrical optics solutions whose leading terms are of the form
\[
(t,x) \longmapsto e^{it|\kappa|^2} e^{\kappa\cdot x}, \qquad \kappa \in \R^n,
\]
supplemented by correction terms depending on the potential. These complex exponentials grow in certain directions.  Using such solutions in an orthogonality relation of the form \eqref{eq:orthogonality_intro} and controlling the corresponding correction terms naturally leads to assuming super-exponential decay of the potential at infinity.

However, in the time-independent case considered here, we work with time-harmonic solutions whose main term is given in the form
\[
(t,x) \longmapsto e^{-i|\kappa|^2 t} e^{-i\kappa\cdot x}, \qquad \kappa \in \R^n,
\]
ignoring again the perturbative terms correcting for the potential. Now the exponents are imaginary, so the leading term is purely oscillatory and unimodular. This allows us to restrict attention to a more regular class of solutions, which however is still sufficient to establish uniqueness in the case of time-independent potentials.

Inverse problems associated with the dynamical Schrödinger equation have been extensively studied; see, for example, \cites{zbMATH01886353,zbMATH06733553,zbMATH05549395,zbMATH05655673,zbMATH05839237,zbMATH06864429}. Many of these works have also considered time-dependent Hamiltonians \cites{zbMATH06769718,zbMATH06516179,zbMATH05379127,zbMATH07033617,zbMATH07242805}. A common feature of these investigations is the use of a dynamical Dirichlet-to-Neumann map, recorded on the boundary of a domain that contains the non-constant portions of the Hamiltonian. 

This formulation stands in contrast to the one presented in our work. Here, the variable part of the Hamiltonian is not localized to a bounded region but is possibly present throughout the whole space. Moreover, our inverse problem is formulated with a distinct data requirement: knowledge only of the system's initial state and its corresponding state at a final time.

The rest of the paper is organized as follows. Section~\ref{sec:aux} collects background material and auxiliary analytic tools, including well–posedness results and resolvent estimates. In Section~\ref{sec:timeindependent_solutions} we construct stationary states for the Schr\"odinger equation with time-independent potentials, treating both the non-endpoint case $q>n/2$ and the endpoint case $q=n/2$ in dimensions $n\ge3$. Section~\ref{sec:orthorelation} is devoted to extending the Alessandrini-type orthogonality relation beyond physical solutions, and in particular to stationary-state solutions. Finally, in Section~\ref{sec:unbounded_endpoint} we prove Theorem~\ref{th:Lq} by applying the extended orthogonality relation to the stationary states constructed earlier.

\section{Preliminaries: Function spaces, Resolvent and Strichartz  estimates} \label{sec:aux} In this section we collect the basic functional-analytic tools and estimates used throughout the paper.

Throughout the paper, we use the Fourier transform on $\R^n$ defined for $f\in L^1(\R^n)$ by
\[
\widehat f(\xi) \coloneqq \frac{1}{(2\pi)^{n/2}} \int_{\R^n} f(x) e^{-i x\cdot \xi}\, \dd x,
\qquad \xi\in\R^n,
\]
and extended in the usual way to tempered distributions, in particular to $L^p(\R^n)$ for $1<p\le2$.

For $p,r\in[1,\infty]$ and an open time interval $I\subset\R$, we use the Banach spaces $L^r\left (I;L^p(\R^n)\right )$ and $C\left (\overline{I};L^p(\R^n)\right )$, equipped with the norms
\begin{equation}\label{eq:Bvalued_norms}
\|u\|_{L^r\left(I;L^p(\R^n)\right)} \coloneqq
\left( \int_I \|u(t,\cd)\|_{L^p(\R^n)}^r\,\dd t \right)^{1/r}, \qquad \|u\|_{C\left(\overline{I};L^p(\R^n)\right)} \coloneqq \sup_{t\in I}\|u(t,\cd)\|_{L^p(\R^n)}.
\end{equation}

We will also use intersections of mixed-norm spaces of the form
\[
L^{r_1}\left (I;L^{p_1}(\R^n)\right )\cap L^{r_2}\left(I;L^{p_2}(\R^n)\right), \textnormal{ and } L^{p_1}(\R^n)\cap L^{p_2}(\R^n)
\]
endowed with the norm
\[
\max\left( \|u\|_{L^{r_1}\left(I;L^{p_1}(\R^n)\right)}, \|u\|_{L^{r_2}\left(I;L^{p_2}(\R^n)\right)}\right) , \textnormal{ and } \max\left( \|u\|_{L^{p_1}(\R^n)}, \|u\|_{L^{p_2}(\R^n)}\right)
\]
respectively.

\subsection{The Kenig--Ruiz--Sogge resolvent estimate}\label{sec:KRS} For $\lambda>0$, we consider the Helmholtz equation
\begin{equation}\label{eq:Helm}
(\Delta+\lambda^2)u=f \qquad \text{in } \R^n.
\end{equation}
For $f\in\mathcal S(\R^n)$, we choose the associated solution operator defined by
\begin{equation}\label{eq:soloper}
P_\lambda f(x) \coloneqq \frac{1}{(2\pi)^{n/2}} \,\mathrm{p.v.}\!\int_{\R^n} \frac{e^{ix\cdot\xi}}{\lambda^2-|\xi|^2} \widehat f(\xi)\,\dd\xi, \qquad x\in\R^n.
\end{equation}
Then $u\coloneqq P_\lambda f$ is a distributional solution of \eqref{eq:Helm}.

We denote by $q_n$ the Lebesgue exponent arising from the Tomas--Stein extension theorem, \cites{stein,Tomas}, defined by
\begin{equation}\label{eq:extLq}
\frac{1}{q_n} \coloneqq \frac{1}{2}-\frac{1}{n+1}=\frac{n-1}{2(n+1)}, \qquad n\geq 2.
\end{equation}
This exponent appears naturally in the analysis of the Helmholtz equation \eqref{eq:Helm} below, and in particular in the $L^p$-bounds for its solution operator.

Let $p_n$ denote the Sobolev exponent for the embedding $\dot H^1(\R^n)\hookrightarrow L^{p_n}(\R^n)$, namely
\[
\frac{1}{p_n}=\frac12-\frac1n=\frac{n-2}{2n}, \qquad n\geq 3.
\]
The following quantitative estimate for $P_\lambda$ is obtained from the work of Kenig, Ruiz, and Sogge \cite{zbMATH04050093} by rescaling:
\begin{equation}\label{eq:KRS}
\|P_\lambda f\|_{L^p(\R^n)} \lesssim \frac{1}{\lambda^{2n(\frac1p-\frac1{p_n})}} \|f\|_{L^{p'}(\R^n)} \quad\text{for}\quad
\begin{cases} 
q_n\leq p\le p_n, & n\geq3, 
\vspace{.6em}
\\ 
q_n\leq p<\infty, & n=2,
\end{cases}
\end{equation}
where the implicit constant depends only on $p$ and $n$. From here on, $p'$ denotes the H\"older conjugate exponent of $p$.

\subsection{Strichartz pairs and well-posedness}\label{subsec: Stri_WP} We assume $n\geq 2$ throughout the paper. We will say that the pair $(r, p)$ is a \emph{Strichartz pair} if 
\[
\frac{2}{r} + \frac{n}{p} = \frac{n}{2}, \quad \textup{with}\quad r, p \in [2, \infty] \quad  \mathrm{and} \quad (n, r, p) \neq (2, 2, \infty).
\]
Note that, in dimension $n=2$, the endpoint pair $(2,\infty)$ is not admissible.

We consider the Schr\"odinger equation in $\Sigma\coloneqq [0,T]\times \R^n$ with a time-independent potential $V=V(x)$ and a force term $F=F(t,x)$ as follows
\begin{equation}\label{eq:schfull}
\begin{cases}
(i\partial_t +\Delta)u-Vu=F,\quad &(t,x)\in \Sigma,
\vspace{.6em}
\\
u(0,\cd)=f, \quad &\textrm{in}\quad \R^n.
\end{cases}
\end{equation}
Assume that the potential $V$ is time-independent and satisfies $V\in L^q(\R^n)$, where $q\geq n/2$ if $n\ge3$ and $q>1$ if $n=2$. Let $(r,p)$ and $(\tilde r,\tilde p)$ be Strichartz pairs satisfying $\frac{1}{q}+\frac{2}{p}=1$. Then, for every $f\in L^2(\R^n)$ and $F\in L^{\tilde r'}((0,T);L^{\tilde p'}(\R^n))$, the initial value problem \eqref{eq:schfull} admits a unique solution
\[
u\in C\left([0,T];L^2(\R^n)\right)\cap L^r\left((0,T);L^p(\R^n)\right);
\]
see Figure~\ref{fig:str}. This follows from the existence and uniqueness results for critical potentials proved by Ruiz--Vega \cite{zbMATH00179225} and Ionescu--Kenig \cite{zbMATH02204588}.

Throughout the paper, the integrability exponents for the potential and the solutions are linked by the H\"older relation
\[
\frac{1}{q}+\frac{1}{p}+\frac{1}{p}=1.
\]
Under this assumption, the following bounds hold:
\begin{equation}\label{eq:holder}
\left|\int_{\R^n} V u v\right| \leq \|V\|_{L^q(\R^n)} \|u\|_{L^p(\R^n)}\|v\|_{L^p(\R^n)},\qquad \|Vu\|_{L^{p'}(\R^n)}\leq \|V\|_{L^q(\R^n)}\|u\|_{L^p(\R^n)}.
\end{equation}
If, in addition, $r\geq 2$, then for any time interval $(0,T)$ one also has
\begin{equation}\label{eq:holder1}
\|Vu\|_{L^{r'} \left((0,T);L^{p'}(\R^n)\right)} \leq T^{\frac{1}{r'}-\frac{1}{r}} \|V\|_{L^q(\R^n)} \|u\|_{L^r \left((0,T);L^p(\R^n)\right)}.
\end{equation}

\begin{remark}\label{rmrk:VuinLp'} Let $V\in L^q(\R^n)$ be time independent, with $q>1$ if $n=2$ or $q\geq n/2$ if $n\geq3$, and let $(r,p)$ be a Strichartz pair satisfying $\frac1q+\frac2p=1$. If $u\in C([0,T];L^2(\R^n))\cap L^r((0,T);L^p(\R^n))$ solves \eqref{eq:schfull} with forcing term $F\in L^{r'}((0,T);L^{p'}(\R^n))$, then \eqref{eq:holder1} implies
\[
(i\partial_t+\Delta)u = F+Vu \in L^{r'}\big((0,T);L^{p'}(\R^n)\big).
\]
In particular, this holds for the endpoint Strichartz pair $(r_n,p_n)=(2,\frac{2n}{n-2})$ when $n\ge3$. This property will be used repeatedly in Section~\ref{sec:orthorelation} and uses the fact that $V$, or the inverse problem itself, is only considered throughout a time interval of finite length, $T<\infty$.
\end{remark}

\begin{figure}[htb]
\centering
\def\svgwidth{300pt}
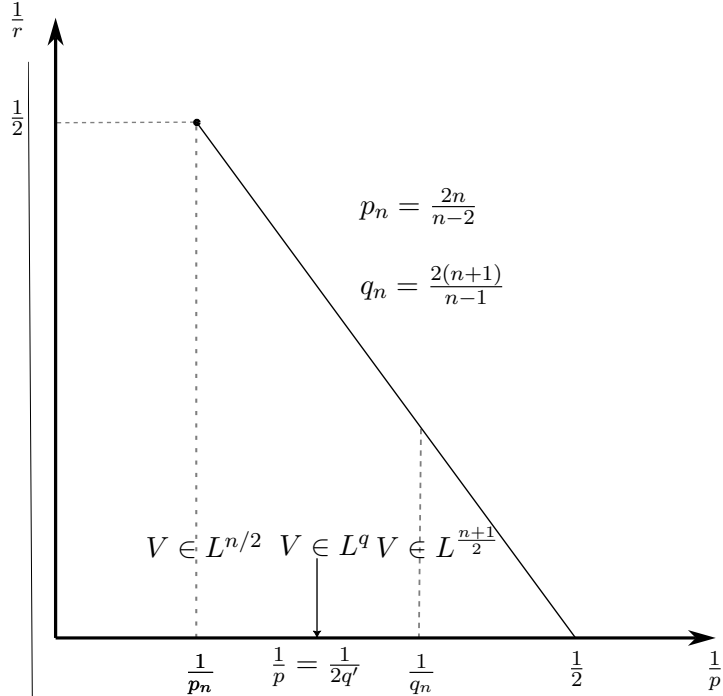
\vspace{-5em}
\caption[Strichartz line]{\small The exponents for the space of potentials $L^q(\R^n)$, the Strichartz pairs $(r,p)$, and the Stein-Tomas extension estimate $L^2(\mathbb S^{n-1})\to L^{p}(\R^n)$ for $p\geq q_n$. The Kenig--Ruiz--Sogge estimate is valid in the range $p\in [q_2,\infty)$ when $n=2$ and $p\in[q_n,p_n]$ when $n\geq 3$.} \label{fig:str}
\end{figure}

\begin{remark}\label{rmrk:numer} A brief comment on the role of the exponents is in order. The Kenig--Ruiz--Sogge estimate \eqref{eq:KRS} holds in the range
\[
q_n \le p \leq p_n \quad \text{if } n\geq3, \qquad q_n \le p < \infty \quad \text{if } n=2.
\]
Imposing the H\"older relation $\frac{2}{p}+\frac{1}{q}=1$, this corresponds to
\[
\frac{n}{2} \leq q \leq \frac{n+1}{2} \quad \text{if } n\geq3, \qquad 1< q \leq \frac{3}{2} \quad \text{if } n=2.
\]
In dimensions $n\ge3$, the endpoint $q=n/2$ corresponds to $p=p_n$, and in this case \eqref{eq:KRS} no longer yields decay in the parameter $\lambda$, which necessitates a separate treatment based on a refinement of \eqref{eq:KRS}, developed in Section~\ref{sec:timeindependent_solutions} below. By contrast, the upper endpoint $q=(n+1)/2$, corresponding to $p=q_n$, does not pose any difficulty in our setting. Indeed, since throughout the paper we assume $V\in L^1(\R^n)\cap L^q(\R^n)$, any potential with $q>(n+1)/2$ automatically belongs to $L^1(\R^n)\cap L^{(n+1)/2}(\R^n)$, with the norm bound
\[
\|V\|_{L^1(\R^n)\cap L^{\frac{n+1}{2}}(\R^n)} \leq \|V\|_{L^1(\R^n)\cap L^q(\R^n)}.
\]
Thus, without loss of generality, we may always restrict attention to the range $q\leq (n+1)/2$.
\end{remark}

\section{Solutions of the time-independent Schr\"odinger equation}\label{sec:timeindependent_solutions}
In this section we construct the special solutions, called stationary states, that will be used in the orthogonality relation~\eqref{eq:orthogonality_intro} to prove Theorem~\ref{th:Lq} in Section~\ref{sec:unbounded_endpoint}. Specifically, we consider functions $\psi$ solving the Schr\"odinger equation
\begin{equation}\label{sch_timeind}
(i\partial_t +\Delta-V)\psi =0 \qquad\mathrm{in}\quad \R\times \R^n,\quad n\geq 2,
\end{equation}
of the form
\begin{equation}\label{eq:steadysol}
\psi_V ^{\lambda,\omega}(t,x)\coloneqq e^{-i\lambda^2 t}(e^{-i\lambda \omega\cdot x}+w^{\mathrm{cor}} (x))\eqqcolon e^{-i\lambda^2 t}\big(w^{(0)}(x)+w^{\mathrm{cor}}(x)\big),\qquad  (t,x)\in\R\times \R^n,
\end{equation}
where $\omega\in\mathbb S^{n-1}$ and  $\lambda>0$. Of course, the functions $w^{(0)},w^{\mathrm{cor}}$ also depend on $\lambda,\omega$ and $V$ but we suppress this dependence in order to simplify the notation. We will also use the notation
\[
\psi_V ^{\lambda,\omega}(t,x)=e^{-i\lambda^2 t}w(x),\qquad w(x)\coloneqq w^{(0)}(x)+w^{\mathrm{cor}}(x),\qquad (t,x)\in\R\times \R^n.
\]

In order for $\psi$ to solve \eqref{sch_timeind} we check that $w^{\mathrm{cor}}$ has to satisfy
\[
\left(\Delta+\lambda^2-V\right)(w^{(0)}+w^{\mathrm{cor}})=0\qquad\mathrm{in}\quad\R^n,
\]
or equivalently
\[
\left(\Delta+\lambda^2\right)w^{\mathrm{cor}}=Vw^{(0)}+Vw^{\mathrm{cor}}\qquad\mathrm{in}\quad\R^n.
\]
By the discussion in \S\ref{sec:KRS} and with the notation therein we note that the above partial differential equation will be satisfied if
\[
w^{\mathrm{cor}}=P_\lambda(Vw^{(0)}+Vw^{\mathrm{cor}}),
\]
where $P_\lambda$ is the solution operator defined in \eqref{eq:soloper}. Such $w^{\mathrm{cor}}$ would thus satisfy
\[
\left(\mathrm{Id}-P_\lambda \circ V\right)w^{\mathrm{cor}}
= P_\lambda(Vw^{(0)}).
\]
We will therefore be able to construct $w^{\mathrm{cor}}$ once we manage to invert the operator $\mathrm{Id}-P_\lambda\circ V$ in $\mathcal L(X^*)$ for a suitable Banach space $X^*$.

Throughout this section we assume that
\[
V\in L^1(\R^n)\cap L^q(\R^n),\qquad  \frac n2 \leq q\leq \frac {n+1}{2};
\]
see Remark~\ref{rmrk:numer} concerning the range of $q$. Since $w^{(0)}$ is unimodular, we have $Vw^{(0)}\in L^1(\R^n)\cap L^q(\R^n)$. Consequently, in order to construct $w^{\mathrm{cor}}$ we will choose a Banach space $X^*$ such that
\[
P_\lambda\left(L^1(\R^n)\cap L^q(\R^n)\right)\subseteq X^*,
\]
and for which $\mathrm{Id}-P_\lambda\circ V$ can be inverted.

We henceforth divide the discussion and the corresponding choice of $X^*$ according to whether $q>n/2$ for $n\geq 2$, corresponding to the non-endpoint case, or $q=n/2$ for $n\geq 3$, corresponding to the critical case.

\subsection{Non-endpoint solutions for \texorpdfstring{$q>n/2$}{q>n/2}} Consider a \emph{time-independent} potential $V\in L^1(\R^n)\cap L^q(\R^n)$ for some $n/2<q\leq (n+1)/2$, and define  $p$ by the H\"older relation
\[
\frac{1}{q}+\frac{2}{p}=1.
\]
We begin by noting that $1<p'<q$ and thus $L^1(\R^n)\cap L^q(\R^n)\subset L^{p'}(\R^n)$. Remembering the mapping properties of $P_\lambda$ from \eqref{eq:KRS} we have that 
\[
P_\lambda\left(L^1(\R^n)\cap L^q(\R^n)\right)\subset P_\lambda\big (L^{p'}(\R^n)\big)\subset L^p(\R^n).
\]
Thus we will be able to construct the correction term $w^{\mathrm{cor}}$ in $L^p(\R^n)\eqqcolon X^*$ once we manage to show that the operator $\mathrm{Id}-P_\lambda\circ V$ is invertible in $\mathcal L(X^*)=\mathcal L(L^p(\R^n))$. This is the content of the following lemma.

\begin{lemma}\label{lem:I-P inverse in Lp}\sl Let $n\geq 2$ and $V \in L^q(\R^n) $ for some $n/2< q \leq (n+1)/2$. Define $p$ via $\frac{1}{q} = 1 - \frac{2}{p} $. There exists a constant $\lambda_V$ depending on $n,q$ and $\|V\|_{L^q(\R^n)}$ such that for every $ \lambda > \lambda_V $, the operator $ \Id - P_\lambda \circ V $ has a bounded inverse in $\mathcal{L}(L^p(\R^n))$. Moreover, if this inverse is denoted by $ ( \Id - P_\lambda \circ V )^{ -1 }$ we have that
\[
 \left\| ( \Id - P_\lambda \circ V )^{-1} \right\|_{\mathcal{L}(L^p(\R^n))} \leq  2 
\]
for all $\lambda \geq 2\lambda_V $.
\end{lemma}

\begin{proof} We just note that under the assumptions on $q$ we have that $p=2q' \in[q_n,p_n)$;
hence the Kenig--Ruiz--Sogge estimate \eqref{eq:KRS} applies. Thus
\[
\left \|(P_\lambda \circ V )F\right\|_{L^p(\R^n)}\lesssim \frac{1}{\lambda^{2n\left(\frac1p-\frac{1}{p_n}\right)}} \|VF\|_{L^{p'}(\R^n)}\leq \frac{1}{\lambda^{n\left(\frac2n-\frac{1}{q}\right)}}\|V\|_{L^q(\R^n)}\|F\|_{L^p(\R^n)}
\]
by \eqref{eq:holder}, the definition of $p_n$ and the relation between $p,q$. Since $V\in  L^q(\R^n)$ it follows that $\|P_\lambda \circ V\|_{\mathcal L(L^p(\R^n))}<1$ if $\lambda>\lambda_V$ for sufficiently large $\lambda_V$ as in the statement of the lemma. The proof is now completed by a standard argument involving Neumann series.
\end{proof}

Returning to the solutions in \eqref{eq:steadysol}, we can now use the invertibility of $\mathrm{Id}-P_\lambda\circ V$ to complete the construction of the correction term $w^{\mathrm{cor}}$. In the statement of the corollary below, the additional assumption $V\in L^1(\R^n)$ is used only to ensure that $Vw^{(0)}\in L^{p'}(\R^n)$, so that $P_\lambda(Vw^{(0)})$ is well defined and belongs to $L^p(\R^n)$; the invertibility of $\mathrm{Id}-P_\lambda\circ V$ on $L^p(\R^n)$ relies solely on the $L^q$ assumption.
\begin{corollary}\label{cor:solutions}\sl Let $n\geq 2$ and consider a time-independent potential $V\in L^1(\R^n)\cap L^q(\R^n)$ for some $n/2<q\leq (n+1)/2$, and let $p$ be defined via $\frac{1}{q} = 1 - \frac{2}{p} $. For $\lambda > 2\lambda_V$ with $\lambda_V$ as in the statement of Lemma~\ref{lem:I-P inverse in Lp}, the functions
\[
w(x)\coloneqq e^{-i\lambda\omega\cdot x}+w^{\mathrm{cor}}(x)= w^{(0)}(x)+w^{\mathrm{cor}}(x),\qquad w^{\mathrm{cor}}(x)\coloneqq (\mathrm{Id}-P_\lambda\circ V)^{-1}[P_\lambda(Vw^{(0)})],
\]
with $ \omega\in\mathbb S^{n-1}$ and  $x\in\R^n$,
are weak solutions of the time-independent Schr\"odinger equation 
\[
(\Delta+\lambda^2-V)w=0\quad\mathrm{in}\quad\R^n.
\]
Additionally, for $\lambda>2\lambda_V$ we have the estimate
\[
\left \|w^{\mathrm{cor}}\right\|_{L^p(\R^n)}\lesssim \frac{1}{\lambda^{2n\left(\frac1p-\frac{1}{p_n}\right)}} \|V\|_{L^{p'}(\R^n)}\leq \frac{1}{\lambda^{2n\left(\frac1p-\frac{1}{p_n}\right)}} \|V\|_{L^1(\R^n)\cap L^q(\R^n)}.
\]
In particular, the function 
\[
\psi_V ^{\lambda,\omega}(t,x)\coloneqq e^{-i\lambda^2 t}w(x)= e^{-i\lambda^2 t}(w^{(0)}(x)+w^{\mathrm{cor}}(x)),\qquad (t,x)\in\R\times \R^n,
\]
is a weak solution of \eqref{sch_timeind}.
\end{corollary}

\subsection{Endpoint solutions in the critical case \texorpdfstring{$q=n/2,\, n\geq 3$}{q=n/2, n>=3}}\label{sec:stend} Here we assume that the potential $V\in L^1(\R^n)\cap L^{n/2}(\R^n)$. We point out that the decay in $\lambda$ in the $L^{p'}(\R^n)\to L^p(\R^n)$ estimate for $P_\lambda$ from \eqref{eq:soloper} was critical in the proof of Lemma~\ref{lem:I-P inverse in Lp} above, and thus also the subsequent Corollary~\ref{cor:solutions}. A second element used in the proof of Lemma~\ref{lem:I-P inverse in Lp} is that multiplication by $V$ is bounded from $L^p(\R^n)$ to $L^{p'}(\R^n)$, as can be read from \eqref{eq:holder} and used in the Neumann series argument.

Isolating the ingredients above, it will suffice to construct a space $X^*=X^* _\lambda \supseteq P_\lambda(L^1(\R^n)\cap L^{n/2}(\R^n))$ depending on $\lambda$, such that the following hold
\begin{align}\label{eq:modRKS}
\|P_\lambda f\|_{X^*}&\leq c(\lambda) \|f\|_{X},      
\\ \label{eq:multV}
\left|\int_{\R^n} V f g\right|&\leq \tilde c(\lambda) \|f\|_{X^*} \|g\|_{X^*},
\end{align}
with $\lim_{\lambda\to +\infty}c(\lambda)\tilde c(\lambda)=0$.

The first estimate above is a modified Kenig--Ruiz--Sogge estimate with some constant $c(\lambda)$. The second estimate  tells us that the operator of multiplication by $V$ is bounded from $X^*$ to $X$ with constant $\tilde c(\lambda)$.  Assuming these two estimates for a moment we have for any $F\in X^*$ that
\[
\left\|(P_\lambda \circ V)F\right\|_{X^*} \leq c(\lambda) \|VF\|_{X}\lesssim c(\lambda)\tilde c(\lambda)\|F\|_{X^*}
\]
and the operator norm of $P_\lambda\circ V$ becomes smaller than $1$ as $\lambda\to \infty$, allowing for the Neumann series argument.

In the remainder of this subsection we construct the space $X^*$ and prove the estimates \eqref{eq:modRKS} and \eqref{eq:multV}. To that end let us revisit the estimate of Kenig, Ruiz and Sogge, \eqref{eq:KRS}. As discussed in Remark~\ref{rmrk:numer}, the case of critical potential corresponds to $n\geq 3$ and  $q=n/2$ which happens exactly when $p=p_n$, since $\frac{2}{p}+\frac{1}{q}=1$. In this case, the estimate for the solution operator $P_\lambda$ in \eqref{eq:soloper} has no decay in $\lambda$. Since $2n(\frac{1}{q_n}-\frac{1}{p_n})=\frac{2}{n+1}$, the endpoint cases of \eqref{eq:KRS} can be written in the form
\[
      \|P_\lambda f\|_{L^{p_n}(\R^n)}\lesssim \|f\|_{L^{p_n '}(\R^n)},\qquad
     \lambda^{\frac{1}{n+1}} \|P_\lambda f\|_{L^{q_n}(\R^n)}\lesssim \frac{1}{\lambda^\frac{1}{n+1}}\|f\|_{L^{q_n '}(\R^n)}.
\]
We will construct the space $X^*$ so that it suitably combines the two estimates above. More precisely define the Banach space
\[
\X_\lambda \coloneqq L^{q_n'}(\R^n)+L^{p_n'}(\R^n) ,\qquad  \| f \|_{\X_\lambda} \coloneqq \inf \left\{ \lambda^{- \frac{1}{n+1}} \| g\|_{L^{q_n'}(\R^n)} + \| h\|_{L^{p_n'}(\R^n)} : \, f = g+h \right\}.
\]
We recall that the dual of $X_\lambda$ is isometrically isomorphic to the space
\[
X_{\lambda} ^*=  L^{q_n}(\R^n) \cap L^{p_n}(\R^n),\qquad \|g\|_{X_\lambda ^*} =\max\left\{ \lambda^{\frac{1}{n+1}} \| g \|_{L^{q_n}(\R^n)}, \| g \|_{L^{p_n}(\R^n)} \right\} .      
\]
 With $X_\lambda$ as our best candidate for the space $X$ let us first verify some easy facts. Since $1<p_n^\prime <q_n^\prime<n/2$ we have that $ L^1(\R^n)\cap L^{n/2}(\R^n) \subset L^{q_n'}(\R^n)$ and $ L^1(\R^n)\cap L^{n/2}(\R^n) \subset L^{p_n'}(\R^n)$, hence $L^1(\R^n)\cap L^{n/2}(\R^n)\subset X_\lambda$. Thus, if we manage to establish~\eqref{eq:modRKS}, then we obtain that
\[
P_\lambda \big(L^1(\R^n)\cap L^{n/2}(\R^n)\big)\subset P_\lambda(X_\lambda ) \subset X_\lambda ^*
\]
as desired. 

\begin{remark}\label{rem:endpoint_norm_for_large_V} Let $\lambda>0$ and $V\in L^{n/2}(\R^n)$. Define 
\[
E_\lambda \coloneqq \left\{ x\in\R^n:\, |V(x)| > \lambda \|V\|_{L^{n/2}(\R^n)} \right\}.
\] 
Note that $V \ind_{\R^n \setminus E_\lambda} \in L^{\frac{n+1}{2}}(\R^n)$ with the estimate
\[
\left \|V \ind_{\R^n \setminus E_\lambda} \right\|_{L^{\frac{n+1}{2}}(\R^n)} \leq \lambda^{\frac{1}{n+1}} \| V \| _{L^{\frac{n}{2}}(\R^n)} .
\]
Defining
\[
\|V\|_{\lambda}\coloneqq \left \|V\ind_{E_\lambda}\right \|_{L^{n/2}(\R^n)}+\lambda^{-\frac{2}{n+1}}\left\|V \ind_{\R^n \setminus E_\lambda} \right\|_{L^{\frac{n+1}{2}}(\R^n)}
\]
and using the previous estimate, the absolute continuity of the integral of $|V|^{n/2}$ together with the fact the $|E_\lambda|\to 0$ as $\lambda\to \infty$, we conclude that $\|V\|_\lambda \to 0$ as $\lambda\to \infty$.
\end{remark} 

The following lemma is the analogue of \eqref{eq:holder}, namely a H\"older-type inequality in the spirit of the desired \eqref{eq:multV}.

\begin{lemma}\sl \label{lem:multiplication_V_endpoint_X_norm} If $ V\in  L^{n/2}(\R^n) $ then we have the H\"older-type inequality 
\[
\left |\int_{\R^n} V f g \right| \leq \|V\|_\lambda   \| f \|_{\X^*_\lambda} \| g \|_{\X^*_\lambda} 
\]
for all $f,g\in \X^*_\lambda$.  Consequently, we have that 
\[
\| V f \|_{\X_\lambda} \leq \|V\|_\lambda \| f \|_{\X_\lambda^*}.
\]
\end{lemma}

\begin{proof}Since $1- \frac{2}{p_n} = \frac{2}{n} $ and $ 1 - \frac{2}{q_n} = \frac{2}{n+1} $, using \eqref{eq:holder} and Remark~\ref{rem:endpoint_norm_for_large_V} we get that
\[
\begin{split}
\Big| \int_{\R^n} V fg \Big| &\leq \Big| \int_{\R^n} \ind_{E_\lambda} V fg \Big| + \Big| \int_{\R^n} \ind_{\R^n \setminus E_\lambda} V fg \Big| 
\\
&\leq \| V \ind_{E_\lambda} \|_{L^{n/2}(\R^n)} \| f\|_{L^{p_n}(\R^n)} \| g\|_{L^{p_n}(\R^n)} 
+ \|V \ind_{\R^n \setminus E_\lambda} \|_{L^{\frac{n+1}{2}}(\R^n)} \| f \|_{L^{q_n}(\R^n)} \| g \|_{L^{q_n}(\R^n)} 
\\
&\leq \|V\|_\lambda\|f \|_{\X^\ast_\lambda} \| g \|_{\X^\ast_\lambda}
\end{split}
\]
as desired.
\end{proof}

The spaces $X_\lambda$ and $X_\lambda^*$ are designed so that $P_\lambda\circ V$ regains a smallness property at the critical exponent $q=n/2$. This yields the promised refinement of \eqref{eq:KRS} in the present setting. We now turn to the verification of \eqref{eq:modRKS} for $X_\lambda^*$.

\begin{lemma}\sl \label{lem:modRKS} Let $n\geq 3$ and $\lambda>0$. For $X_\lambda$ as defined above, there holds
\[
\|P_\lambda f\|_{X_\lambda ^*} \lesssim \|f\|_{X_\lambda}.
\]
\end{lemma}

\begin{proof} The conclusion of the lemma is equivalent to the four estimates below
\[
\begin{split}
\lambda^{\frac{1}{n+1}}\|P_\lambda f\|_{L^{q_n}(\R^n)}\lesssim  \lambda^{-\frac{1}{n+1}}\|f\|_{L^{q_n '}(\R^n)},&\qquad \|P_\lambda f\|_{L^{p_n}(\R^n)}\lesssim \|f\|_{L^{p_n '}(\R^n)},
\\
\lambda^{\frac{1}{n+1}}\|P_\lambda f\|_{L^{q_n}(\R^n)}\lesssim \|f\|_{L^{p_n '}(\R^n)},&
\qquad \|P_\lambda f\|_{L^{p_n}(\R^n)}\lesssim  \lambda^{-\frac{1}{n+1}}\|f\|_{L^{q_n '}(\R^n)}.
\end{split}
\]
The estimates in the first line of the display above are exactly the endpoint cases of \eqref{eq:KRS}. The two estimates in the second line are dual to each other so it will suffice to prove
\begin{equation}\label{eq:RKSsdual}
      \lambda^{\frac{1}{n+1}}\|P_\lambda f\|_{L^{p_n}(\R^n)}\lesssim \|f\|_{L^{q_n '}(\R^n)}.
\end{equation}

We turn to this task now. Begin by choosing a smooth bump function $\vp$ on $\R^n$ with $0\leq \vp\leq 1$, $\vp \equiv 1$ in $\{ x\in \R^n : \, |x|<2 \}$ and $\supp \vp \subset \{x\in\R^n: \, |x|< 4\}$. We fix $f\in \mathcal S(\R^n)$ and define the smooth frequency projections
\[
\wh{S_{<\lambda}f}(\xi) \coloneqq \vp ( \xi/\lambda ) \wh{f}(\xi),\qquad \wh{S_{>\lambda}f}(\xi) \coloneqq [ 1 - \vp ( \xi/\lambda ) ] \wh{f}(\xi),\qquad \xi\in\R^n.
\]
Let us write also $S_k f \coloneqq S_{<2^{k+1}}f-S_{<2^k}f$ for $k\in\Z$. We estimate
\begin{equation}\label{eq:resolvent_projections_split}
\|P_\lambda f \|_{L^{p_n}(\R^n)} \leq \| P_\lambda(S_{<\lambda} f) \|_{L^{p_n}(\R^n)} + \| P_\lambda(S_{>\lambda}f) \|_{L^{p_n}(\R^n)}.
\end{equation}
Since $P_\lambda(S_{<\lambda}f)$ is supported in frequency in $B(0,4\lambda)$ and
$p_n>q_n$, Bernstein's inequality gives
\begin{align*}
 \| P_\lambda(S_{<\lambda} f) \|_{L^{p_n}(\R^n)} &\lec (\lambda^n)^{\frac{1}{q_n} - \frac{1}{p_n}}  \| P_\lambda(S_{<\lambda} f) \|_{L^{q_n}(\R^n)} \lec \lambda^{\frac{1}{n+1}} \| P_\lambda f \|_{L^{q_n}(\R^n)}\lesssim \lambda^{- \frac{1}{n+1}} \| f \|_{L^{q_n'}(\R^n)}.
 \end{align*}

For the second summand in \eqref{eq:resolvent_projections_split} note that $|\xi|>2\lambda$ on the frequency support of $\wh{S_{>\lambda} f}$, which means that
$ \| \widehat{P_\lambda(S_k f)} \|_{L^2(\R^n)} = 2^{2k} \| \widehat{S_k f} \|_{L^2(\R^n)}$
whenever $2^k > 2 \lambda$. Then we use again that $q_n ' \leq 2 \leq p_n$ together with Bernstein's inequality, symbol estimates and again Bernstein's inequality, to get  
\[
\begin{split}
 \| P_\lambda (S_{>\lambda} f) \|_{L^{p_n}(\R^n)} &\lec \sum_{2^k > 2 \lambda} 2^{kn(\frac{1}{2} - \frac{1}{p_n})}      \| P_\lambda(S_k f) \|_{L^2(\R^n)} 
\simeq \sum_{ 2^k\gtrsim \lambda} 2^{-k} \| S_k f \|_{L^2(\R^n)} 
\\
&\lec \sum_{2^k \gtrsim \lambda} 2^{-k} 2^{kn(\frac{1}{q_n'} - \frac{1}{2})}  \| S_k f \|_{L^{q_n'}(\R^n)}
\lesssim \lambda^{-\frac{1}{n+1}} \| f \|_{L^{q_n'}(\R^n)}.
\end{split}
\]
This yields the desired estimate for the second summand in \eqref{eq:resolvent_projections_split}. Summing the estimates for the low and high frequencies of $P_\lambda f$ proves \eqref{eq:RKSsdual} and completes the proof of the lemma.
\end{proof}

Combining the previous results and the discussion in this subsection it is now routine to fill in the details of the proof of the following corollary, where the perturbations $w^{\mathrm{cor}}$ are constructed for critical potentials $V\in L^1(\R^n)\cap L^{n/2}(\R^n)$.

\begin{corollary}\label{cor:solutionsVcrit}\sl Let $n\geq 3$ and consider a time-independent potential $V\in L^1(\R^n)\cap L^{n/2}(\R^n)$.  Let $X_\lambda,X_\lambda ^*$ and $\|V\|_\lambda$ as defined above. There exists a constant $\lambda_V$ depending only on $V$ and the dimension such that for $\lambda > \lambda_V$, the functions
\[
w(x)\coloneqq e^{-i\lambda\omega\cdot x} +w^{\mathrm{cor}} (x)\eqqcolon w^{(0)}(x)+w^{\mathrm{cor}}(x),\qquad w^{\mathrm{cor}}(x)\coloneqq (\mathrm{Id}-P_\lambda\circ V)^{-1}[P_\lambda(Vw^{(0)})],
\]
with $\omega\in\mathbb S^{n-1} $ and $ x\in\R^n$,
are weak solutions of the time-independent Schr\"odinger equation 
\[
(\Delta+\lambda^2-V)w=0\quad\mathrm{in}\quad\R^n.
\]
Additionally, the following estimates are satisfied
\[
\left\|w^{\mathrm{cor}} \right\|_{X_\lambda ^*}\lesssim \|V\|_{X_\lambda}\leq \lambda^{-\frac{1}{n+1}} \|V\|_{L^{q_n '}(\R^n)}\leq  \lambda^{-\frac{1}{n+1}} \|V\|_{L^1(\R^n)\cap L^{n/2}(\R^n)}
\]
In particular, the function 
\[
\psi_V ^{\lambda,\omega}(t,x)\coloneqq e^{-i\lambda^2 t}w(x)= e^{-i\lambda^2 t}(w^{\mathrm{cor}}(x)+w^{(0)}(x)),\qquad (t,x)\in\R\times \R^n,
\]
is a weak solution of \eqref{sch_timeind}.
\end{corollary}

\begin{proof} As discussed in the comments preceding this lemma, we will invert the operator $(\mathrm{Id}-P_\lambda \circ V)$ on $\mathcal L(X_\lambda ^*)$. We combine the estimates of Lemma~\ref{lem:multiplication_V_endpoint_X_norm} with the one of Lemma~\ref{lem:modRKS} to estimate for each 
$F\in X_\lambda ^*$ 
\[
\|(P_\lambda \circ V )F\|_{X_\lambda ^*} \lesssim \|VF\|_{X_\lambda}\leq \|V\|_\lambda \|F\|_{X_\lambda ^*}.
\]
Now Remark~\ref{rem:endpoint_norm_for_large_V} shows that $\|P_\lambda \circ V \|_{\mathcal L(X_\lambda ^*)}<\frac{1}{2}$ if $\lambda$ is sufficiently large, say $\lambda>\lambda_V$ for some constant $\lambda_V$ depending on $V$. We can thus invert $\mathrm{Id}-P_\lambda\circ V$ in $\mathcal L(X_\lambda ^*)$ for all $\lambda>\lambda _V$ and $\|(\mathrm{Id}-P_\lambda\circ V)^{-1}\|_{\mathcal L(X_\lambda ^*)}<2$. Consequently, the following chain of estimates holds
\[
\|w^{\mathrm{cor}} \|_{X_\lambda ^*}\lesssim \|P_\lambda(Vw^{(0)})\|_{X_\lambda ^*}\lesssim \|V\|_{X_\lambda}.
\]
In order to complete the proof we just notice that $L^{q_n '}$ embeds continuously into $X_\lambda$ and the following norm estimate holds
\[
\| f\|_{X_\lambda} \leq \lambda^{-\frac{1}{n+1}}\|f\|_{L^{q_n '}}
\]
which yields the desired estimate when applied to $V$. Note that $1<q_n '=\frac{2(n+1)}{n+3}<n/2$ for $n\geq 3$ so the $L^{q_n '}$ norm of $V$ is controlled by $\|V\|_{L^1(\R^n)\cap L^{n/2}(\R^n)}$ and the proof is complete.
\end{proof}

\section{An orthogonality formula for stationary states}\label{sec:orthorelation} The goal of this section is to prove an orthogonality formula in the spirit of \eqref{eq:orthogonality_intro}, but one where the physical solutions $u_1,v_2$ are replaced by stationary state solutions such as those constructed in Section~\ref{sec:timeindependent_solutions} and in particular having the form \eqref{eq:steadysol}. We first need to state and prove a series of integration by parts identities which will be used in extending the orthogonality relation \eqref{eq:orthogonality_intro}.

\subsection{Integration by parts identities} We record several integration by parts formulas that will be used to derive the orthogonality relation for stationary states in Proposition~\ref{pr:int_by_parts_time-harmonic_solutions}.

We first state a special case of \cite{caro2025initialtofinalstateinverseproblemunbounded}*{Proposition 5.2}, adapted to our setting. The proof is omitted, since it follows the same argument; see \cite{caro2025initialtofinalstateinverseproblemunbounded}*{Appendix A} for details of the proof.

\begin{proposition}\label{prp:integration_by_parts_unbounded_no-endpoint}\sl Let $n\geq 2$ and $(r,p)$ be a Strichartz pair. Assume
\[
u,v \in C\left ([0,T];L^2(\R^n)\right) \cap L^r\left((0,T);L^p(\R^n)\right)
\]
be such that
\[
(i\partial_t+\Delta)u, (i\partial_t+\Delta)v \in L^{r'}\big((0,T);L^{p'}(\R^n)\big).
\]
Then
\[
\int_{\Sigma} \left[(i\partial_t+\Delta)u \ovl v - u \ovl{(i\partial_t+\Delta)v}\right] = i \int_{\R^n} \left[u(T,\cd)\ovl{v(T,\cd)}-u(0,\cd)\ovl{v(0,\cd)}\right].
\]
\end{proposition}

We now prove an extension of Proposition~\ref{prp:integration_by_parts_unbounded_no-endpoint} in the special case $u(0,\cd)=u(T,\cd)=0$, where the integration by parts identity remains valid even when the second function is only assumed to lie in $C([0,T];L^{p}(\R^{n}))$.

\begin{proposition}\label{cor:integration_by_parts_unbounded_no-endopint2}\sl
Let $n\geq 2$ and $(r,p)$ be a Strichartz pair. Let
\[
v \in C\left([0,T];L^p(\R^n)\right)
\]
and
\[
u \in C\left([0,T];L^2(\R^n)\right) \cap L^r\left((0,T);L^p(\R^n)\right)\quad\text{with}\quad u(0,\cd)=u(T,\cd)=0,
\]
be such that
\[
\left(i\partial_t+\Delta\right)u, \left(i\partial_t+\Delta\right)v \in L^{r'}\big((0,T);L^{p'}(\R^n)\big).
\]
Then
\[
\int_\Sigma \left(i\partial_t+\Delta\right)u \ovl v = \int_\Sigma u \ovl{\left(i\partial_t+\Delta\right)v}.
\]
\end{proposition}

\begin{proof} If we additionally assume that $v\in C([0,T];L^2(\R^n))$, the conclusion follows directly from Proposition~\ref{prp:integration_by_parts_unbounded_no-endpoint}, since the boundary terms vanish due to the conditions $u(0,\cd)=u(T,\cd)=0$. The main part of the proof is therefore devoted to removing the assumption $v\in C([0,T];L^2(\R^n))$ and reducing the case $v\in C([0,T];L^p(\R^n))$ to the $L^2$-case via approximation.

Let $\chi$ be a smooth bump function on $\R^n$ such that $0\leq \chi \leq 1$, $\chi\equiv 1$ on $\{|x|\leq 1\}$, and $\supp \chi \subset \{|x|\leq 2\}$. Let also $\phi\in \mathcal S(\R^n)$ be a nonnegative bump function with $\supp \phi \subset B(0,1)$, satisfying $\int_{\R^n}\phi=1$. For $\varepsilon>0$, let $R=R(\varepsilon)>0$ be a function such that
\[
\lim_{\varepsilon\to0} R(\varepsilon)=+\infty,
\]
which will be chosen in the course of the proof.

We define the rescaled bump functions $\chi^R(x)\coloneqq \chi(x/R)$ and $\phi_\varepsilon(x)\coloneqq \varepsilon^{-n}\phi(x/\varepsilon)$ and for general $v \in C([0,T];L^p(\R^n))$, we set
\[
v_{R,\varepsilon} \coloneqq \chi^R (\phi_\varepsilon * v).
\]
Here the convolution $\phi_\varepsilon * v$ is taken only in the space variables.  Since $v\in C([0,T];L^p(\R^n))$ with $p\geq 2$ and $\chi^R$ has compact support, it follows that
\[
v_{R,\varepsilon}\in C\left([0,T];L^2(\R^n)\right) \cap L^r\left((0,T);L^p(\R^n)\right).
\]
Now an application of the Leibniz rule in the sense of distributions yields
\begin{equation}\label{eq:Leibniz_distributional}
\left(i \d_t + \Delta \right)v_{R, \ve} = \chi^R \left(i\d_t + \Delta\right)(v*\phi_\ve) +  ( v*\phi_\ve ) \Delta(\chi^R)  + 2 ( \nabla \chi^R ) \cdot \nabla (v*\phi_\ve),
\end{equation}
in $\mathcal{D}' (\Sigma)$. We record for later use the easy identities
\begin{equation}\label{eq:scalederiv}
\begin{split}
& \Delta(\chi^R) = R^{-2} (\Delta \chi)^R \ind_{\{R\leq |x| \leq 2R\}}, \qquad \nabla (\chi^R) = R^{-1} (\nabla \chi)^R \ind_{\{R\leq |x| \leq 2R\}}, 
\\
&  \nabla (\phi_\ve) = \ve^{-1} (\nabla \phi)_{\ve}. 
\end{split}
\end{equation}

We estimate each term in \eqref{eq:Leibniz_distributional}. First, we have
\[
\chi^R (i\d_t + \Delta)(v*\phi_\ve) = \chi^R \phi_\ve * (i\d_t + \Delta)v \in L^{r'}\big((0, T); L^{p'}(\R^n)\big).
\]
Moreover, the function
\[
( v*\phi_\ve ) \Delta(\chi^R) = R^{-2} (\Delta \chi)^R (v*\phi_\ve)
\]
has support contained in $\{x\in\R^n:\, R\leq |x|\leq 2R\}$ by the definition of $\chi$. Since $p\geq 2$ we have that $p'\leq 2 \leq p$, and hence $(v*\phi_\ve) \Delta (\chi^R)\in C([0, T]; L^{p'}(\R^n))$.  Similarly, we also infer that 
\[
( \nabla \chi^R ) \nabla (v*\phi_\ve) = R^{-1} \ve^{-1} (\nabla \chi)^R (v* (\nabla \phi)_\ve) \in C\big ([0, T]; L^{p'}(\R^n)\big). 
\]
Therefore, $(i \d_t + \Delta )v_{R, \ve} \in L^{r'}((0, T); L^{p'}(\R^n))$ and we can use the integration by parts formula of Proposition~\ref{prp:integration_by_parts_unbounded_no-endpoint} with the right hand side equal to zero to get 
\[
\int_\Sigma [(i\d_t + \Delta)u] \ovl{v_{R,\ve}} = \int_\Sigma \ovl{(i\d_t + \Delta) v_{R, \ve}}  u.
\]
Notice that, in order to complete the proof it will be enough to show that 
\begin{equation}\label{eq: Cor_int_by_parts_approx1}
\lim_{\ve \to 0} \int_\Sigma \left[(i\d_t + \Delta) u\right] \ovl{(v_{R, \ve} - v)} = 0, 
\end{equation}
and
\begin{equation}\label{eq: Cor_int_by_parts_approx2}
\lim_{\ve \to 0} \int_\Sigma u  \ovl{\left(i\d_t + \Delta\right)(v_{R, \ve} - v)} = 0.
\end{equation}

We begin with the proof of~\eqref{eq: Cor_int_by_parts_approx1}. We have that 
\begin{align*}
\Big| \int_\Sigma  [(i\d_t + \Delta) u] \ovl{(v_{R, \ve} - v)} \Big| &\leq 
\| (i\d_t + \Delta)u \|_{L^{r'}((0, T); L^{p'}(\R^n))} \| v_{R, \ve} - v \|_{L^r((0, T); L^p(\R^n))} \\
&\leq \| (i\d_t + \Delta)u \|_{L^{r'}((0, T); L^{p'}(\R^n))} \| (\chi^R - 1) v\|_{L^r((0, T); L^p(\R^n))} \\
&+ \| (i\d_t + \Delta)u \|_{L^{r'}((0, T); L^{p'}(\R^n))} \| \phi_\ve * v - v \|_{L^r((0, T); L^p(\R^n))}
\end{align*}
which implies~\eqref{eq: Cor_int_by_parts_approx1} as $\ve \to 0$ by dominated convergence. 

We now turn to the proof of the more challenging limit~\eqref{eq: Cor_int_by_parts_approx2}. We have that 
\begin{equation}\label{eq: Cor_int_by_parts_approx2_split} 
\begin{split}
\left| \int_\Sigma \big[ \ovl{ (i\d_t + \Delta)(v_{R,\ve} -v) } \big] u \right| &\leq
\left| \int_\Sigma \big[ \ovl{ \chi^R [(i\d_t+\Delta)v]*\phi_\ve - (i\d_t+\Delta)v }\big] u \right| 
\\
&\qquad +\left| \int_\Sigma \overline{\Delta(\chi^R) (v*\phi_\ve)} u \right| 
+ 2 \left| \int_\Sigma \overline{\nabla (\chi^R) \cdot \nabla (v*\phi_\ve)} u \right|
\\
&\eqqcolon \mathrm{I}(\ve,R)+\mathrm{II}(\ve,R)+\mathrm{III}(\ve,R).
\end{split}  
\end{equation}

We begin by estimating $\mathrm I(\ve,R)$. We have
\[
\begin{split}
\mathrm{I}(\ve,R) &\leq  \int_\Sigma \left| \left[\ovl{\phi_\ve*\left(i\d_t+\Delta\right)v - \left(i\d_t+\Delta\right)v } \right] \chi^R \right| |u|  + \int_\Sigma \left|\chi^R - 1\right| \left|\left(i\d_t +\Delta\right)v \right| |u| 
\\
&\leq \left\| \phi_\ve*\left(i\d_t+\Delta\right)v - \left(i\d_t+\Delta\right)v \right\|_{L^{r'}((0, T); L^{p'}(\R^n))} \| u \|_{L^r((0, T); L^p(\R^n))} 
\\
& \qquad + \int_\Sigma \left|\chi^R - 1\right| \left|\left(i\d_t +\Delta\right)v \right| |u|.
\end{split}
\]
In the right-hand side of the display above, the first summand tends to $0$ as $\varepsilon\to0$ by the convergence of the approximate identity $\phi_\varepsilon*(i\d_t+\Delta)v$ in $L^{p'}(\R^n)$, while the second summand tends to $0$ as $R\to\infty$ by dominated convergence. Consequently, $\mathrm I(\ve,R(\ve))\to 0$ as $\varepsilon\to0$ with any choice of $R=R(\varepsilon)\to\infty$ as $\ve \to 0$.

We estimate the worst term $\mathrm{III}(\varepsilon,R)$. By the assumptions on $\chi,\phi$ and \eqref{eq:scalederiv}, we have
\[
\left|\nabla(\chi^R)\cdot\nabla(v*\phi_\varepsilon)\right|= \ve^{-1} \left|\nabla(\chi^R)\cdot (v*(\nabla\phi)_\varepsilon)\right| \lesssim R^{-1}\varepsilon^{-1} \ind_{\{R\le |x|\le 2R\}}  \left(|v|*|(\nabla\phi)_\varepsilon|\right).
\]
Since $\supp (\nabla\phi)_\varepsilon \subset B(0,\varepsilon)$, it follows that, provided
$\varepsilon\ll R$,
\[
\ind_{\{R\le |x|\le 2R\}}\left(|v|*|(\nabla\phi)_\varepsilon|\right) \leq \left(\ind_{\{|x|\simeq R\}}|v|\right)*|(\nabla\phi)_\varepsilon|.
\]
Therefore,
\[
\begin{split}
\mathrm{III}(\varepsilon,R) &\lesssim R^{-1}\varepsilon^{-1}\int_\Sigma \left[\left(\ind_{\{|x|\simeq R\}}|v|\right)*|(\nabla\phi)_\varepsilon|\right](t,x) |u(t,x)|\,\dd x\,\dd t
\\
&\le R^{-1}\varepsilon^{-1}  \left\|\left(\ind_{\{|x|\simeq R\}}|v|\right)*|(\nabla\phi)_\varepsilon|\right\|_{L^1((0,T);L^2(\R^n))} \|u\|_{C([0,T];L^2(\R^n))}
\\
&\le R^{-1}\varepsilon^{-1} \|v\|_{L^1((0,T);L^2(\{|x|\simeq R\}))}\|u\|_{C([0,T];L^2(\R^n))} 
\\
&\lesssim R^{-1}\varepsilon^{-1}  R^{n(\frac12-\frac1p)} \|v\|_{L^1((0,T);L^p(\{|x|\simeq R\}))}\|u\|_{C([0,T];L^2(\R^n))} .
\end{split}
\]
Here, we used Cauchy--Schwarz in $x$ and H\"older in $t$ to pass to the second line, Young's inequality in $x$ in passing to the third line, together with $\|(\nabla\phi)_\varepsilon\|_{L^1(\R^n)}\lesssim 1$, and finally H\"older on the annulus $\{|x|\simeq R\}$ to pass from $L^2(\{|x|\simeq R\})$ to $L^p(\{|x|\simeq R\})$ with $p\geq2$. Using the calculation
\[
-1+n\left(\frac12-\frac1p\right)= n\left(\frac1{p_n}-\frac1p\right), \qquad  \frac1{p_n}=\frac12-\frac1n,
\]
and writing $R=R(\varepsilon)$, we have proved
\[
\mathrm{III}(\varepsilon,R(\varepsilon)) \lesssim \varepsilon^{-1}R(\varepsilon)^{-n(\frac1p-\frac1{p_n})} \|v\|_{L^1((0,T);L^p(\{|x|\simeq R(\varepsilon)\}))} \|u\|_{C([0,T];L^2(\R^n))}.
\]

If $p<p_n$, then $n(\frac1p-\frac1{p_n})>0$, and choosing
\[
R(\varepsilon)\coloneqq \varepsilon^{-\frac{2}{n(\frac1p-\frac1{p_n})}}
\]
gives $\mathrm{III}(\varepsilon,R(\varepsilon))\lesssim \varepsilon$, hence $\mathrm{III}(\varepsilon,R(\varepsilon))\to0$ as $\varepsilon\to0$.

In the endpoint case $p=p_n$ we have to be more careful since we get that
\[
\mathrm{III}(\ve,R(\ve))
\leq \frac{1}{\ve}  \|u\|_{C([0, T]; L^2(\R^n))}  \|v\|_{L^1\left((0, T); L^{ p_n}(|x|\simeq R(\ve))\right)} .
\]
Notice that
\[
v\ind_{\{|x|\simeq R\}}\ind _{(0,T)} \in C\left([0,T];L^{p_n}(\R^n)\right) \subset L^1\left((0,T);L^{p_n}(\R^n)\right),
\]
and that
\[
v(t,x)\ind_{\{|x|\simeq R\}}\ind_{(0,T)}(t)\to0 \quad\text{as }R\to\infty
\]
for every $(t,x)\in\Sigma$. By dominated convergence,
\[
\lim_{R\to\infty} \|v\|_{L^1((0,T);L^{p_n}(\{|x|\simeq R\}))}=0.
\]
Hence, for every $\ve>0$ there exists $R(\ve)>0$ which can be chosen to also satisfy $R(\ve)>\ve^{-1}$ such that $\|  v \|_{L^1((0,T); L^{p_n}(\{|x|\simeq R\}))}<\ve ^2$ for all $R\geq R(\ve)$. This shows that in the endpoint case $p=p_n$ one can choose $R(\ve)$, with $\lim_{\ve\to0 }R(\ve)=+\infty$, so that
\[
\mathrm{III}(\ve,R(\ve)) \lesssim \frac{1}{\ve}  \|u\|_{C([0, T]; L^2(\R^n))}  \|v\|_{L^1((0, T); L^{p_n}(|x|\simeq R(\ve)))} \to 0 \quad \text{as}\quad \ve\to 0.
\]

The term $\mathrm{II}(\varepsilon,R(\varepsilon))$ is estimated similarly. In particular, in the endpoint case $p=p_n$ we obtain
\[
\mathrm{II}(\ve,R(\ve)) \lesssim  \frac{1}{R(\ve)} \|u\|_{C([0, T]; L^2(\R^n))} \|v\|_{L^1((0, T); L^{p_n}(\R^n))} \to 0 \quad \textup{as} \quad \ve\to 0. 
\]
Combining the estimates we see that there is a choice $R(\ve)\to \infty$ when $\ve\to 0$ such that 
\[
\lim_{\ve\to 0} \left[\mathrm{I}(\ve,R(\ve))+\mathrm{II}(\ve,R(\ve))+\mathrm{III}(\ve,R(\ve))\right]=0
\]
so the right hand side of~\eqref{eq: Cor_int_by_parts_approx2_split} tends to 0 as $\ve\to 0$. This  shows~\eqref{eq: Cor_int_by_parts_approx2} and completes the proof of  Proposition~\ref{cor:integration_by_parts_unbounded_no-endopint2}. 
\end{proof}

With this in hand we proceed with the following integration by parts formula for complex exponential solutions. 

\begin{lemma}\label{lem:int_by_parts_complex-exp_Lp}\sl Let $(r,p)$ be a Strichartz pair with $p\ge q_n$, and assume in addition that $p\le p_n$ if $n\ge3$, or $p<\infty$ if $n=2$. Let
\[
u\in C\left([0,T];L^2(\R^n)\right )\cap L^r\left((0,T);L^p(\R^n)\right), \qquad u(0,\cd)=u(T,\cd)=0,
\]
and suppose that
\[
(i\partial_t+\Delta)u\in L^1(\Sigma)\cap L^{r'}\big((0,T);L^{p'}(\R^n)\big).
\]
Then
\[
\int_\Sigma e^{i(\lambda^2 t+\lambda\omega\cdot x)}  (i\partial_t+\Delta)u = 0,
\]
for every $\lambda>0$ and every $\omega\in\Sph^{n-1}$.
\end{lemma}

\begin{proof} We first construct an approximation of the function $\psi(t,x)\coloneqq e^{  -i \lambda^2 t - i\lambda \omega \cdot x }$. Our favorite way of doing this is the following. Consider $\chi \in \mathcal S (\R^n)$ such that $0 \leq \chi (x) \leq 1$ for all $x \in \R^n$ and $\supp \chi \subset \{ x \in \R^n : \, |x| < 1/2 \}$. We will also choose $\chi$ so that 
\[
\chi(0)=\frac{1}{(2\pi)^{n/2}}\int_{\R^n} \widehat{\chi}(\eta)\,\dd \eta=1.
\]
For $\varepsilon > 0$, define
\[ 
f^\ve (x) \coloneqq \chi(\ve x) e^{-i \lambda \omega \cdot x}\eqqcolon \chi^\ve (x) e^{-i \lambda \omega \cdot x}, \qquad  x \in \R^n. 
\]
Clearly $f^\ve \in \mathcal S(\R^n)$ for each $\ve>0$. Now let $\psi_\ve: \R\times \R^n\to \C$ be defined as 
\[
\psi_\ve (t,x)\coloneqq e^{it\Delta}(f^\ve)(x)=\frac{1}{(2\pi)^{n/2}}\int_{\R^n} e^{-i|\eta|^2t + i \eta\cdot x} \widehat{f^\ve}(\eta)\, \dd \eta,\qquad (t,x)\in \R\times \R^n.
\]
By construction we have that $(i\partial_t+\Delta)\psi_\ve=0$ in $\R\times \R^n$ and $\psi_\ve \in C^\infty(\R;\mathcal S(\R^n))$. In particular $\psi_\ve\in C([0,T];L^2(\R^n))\cap L^{r}((0,T);L^p(\R^n))$ for any Strichartz pair $(r,p)$. 

Now, consider a function $u$ as in the assumption and apply Proposition~\ref{cor:integration_by_parts_unbounded_no-endopint2} to the functions $u,\psi_\ve$ as above to get
\[
 \int_{\Sigma}  \left(i\partial_t+\Delta\right)u  \ovl{\psi_\ve} = \int_{\Sigma} u  \ovl{\left(i\partial_t+\Delta\right)\psi_\ve}  = 0.     
\]
Noting that $\widehat {f^\ve}(\eta)=\widehat{\chi^\ve}(\eta+\lambda \omega)=\ve^{-n}\widehat{\chi}((\eta+\lambda\omega)/\ve)$, we write for $(t,x)\in \Sigma$
\[
\begin{split}
\psi_\ve (t,x)-\psi(t,x)& = \frac{1}{(2\pi)^{n/2}}\int_{\R^n} \widehat{f^\ve}(\eta) e^{-i|\eta|^2t + i\eta\cdot x}\, \dd \eta - e^{-i\lambda^2 t - i\lambda\omega\cdot x}
\\
&= \frac{1}{(2\pi)^{n/2}}\int_{\R^n}\widehat{\chi^\ve}(\eta+\lambda\omega) e^{-i|\eta|^2t+i\eta\cdot x}\, \dd \eta - e^{-i\lambda^2 t -i\lambda\omega\cdot x}
\\
&= \frac{1}{(2\pi)^{n/2}}\int_{\R^n}\frac{1}{\ve^n}\widehat{\chi}\left(\frac{\xi}{\ve}\right) e^{-i|\xi-\lambda\omega|^2t +i(\xi-\lambda\omega)\cdot x}\, \dd \xi - e^{-i\lambda^2 t - i\lambda\omega\cdot x}
\\
&=\frac{e^{-i\lambda^2 t - i\lambda\omega\cdot x}}{(2\pi)^{n/2}}
\int_{\R^n} \widehat{\chi }(\eta ) \left[e^{-i\left(\ve^2|\eta|^2-2\ve \lambda \omega\cdot \eta\right)t + i\ve\eta\cdot x} - 1\right]\, \dd \eta ,
\end{split}
\]
by our assumption $\chi(0)=1$. Since $\chi\in\mathcal S(\R^n)$, we have $\psi_\ve(t,x)\to \psi(t,x)$ as $\ve \to 0$ for every $(t,x)\in \R\times\R^n$ by dominated convergence. The conclusion now follows by yet another application of dominated convergence since
\[
\left| \ovl{(\psi_\ve-\psi)}{\left(i\partial_t+\Delta\right)u}\right|\leq \left(1+(2\pi)^{-n/2}\| \widehat{\chi} \|_{L^1(\R^n)}\right)\left|\left(i\partial_t+\Delta\right)u\right|\in L^1(\Sigma)
\]
since by our assumption $(i\partial_t+\Delta)u\in L^1(\Sigma)$.
\end{proof}

\subsection{Orthogonality formulas} We now relate the initial-to-final-state maps to the difference of the corresponding potentials via solutions of \eqref{eq:Schrodinger}. When $\mathcal U_T^1=\mathcal U_T^2$, this yields the basic orthogonality relation \eqref{eq:orthogonality_intro} for physical solutions.

\begin{proposition}\label{prp:CR_prop4.1_unbounded}\sl Let $V_1, V_2 \in L^q(\R^n)$ for some $q\in(1,\infty]$ if $n=2$, or for some $q\in[n/2,\infty]$ if $n\geq 3$, and let $(r, p)$ be a Strichartz pair with $\frac{1}{q}+\frac{2}{p}=1$.  For every fixed $T>0$ and every $f, g \in L^2(\R^n)$ we have that 
\[
i\int_{\R^n} (\cal{U}^1_T - \cal{U}^2_T) f\ovl{g} = \int_{\Sigma} (V_1 - V_2) u_1 \ovl{v_2},
\]
where $u_1 \in C([0, T]; L^2 (\R^n)) \cap L^r((0, T); L^p(\R^n))$ is a solution of \eqref{eq:Schrodinger} with potential $V_1$ and initial data $f$, while $v_2 \in C([0, T]; L^2 (\R^n)) \cap L^r((0, T); L^p(\R^n))$ is the physical solution of the following final-value problem 
\begin{equation}\label{eq:FVP}
\begin{cases}
i\d_t v_2 = -\Delta v_2 + \ovl{V_2} v_2 \quad &\mathrm{in} \quad \Sigma,
\vspace{.6em}
\\
v_2(T, \cd) = g  &\mathrm{in} \quad \R^n.
\end{cases}
\end{equation}
\end{proposition}

We omit the proof of the proposition above since it follows exactly the same lines as~\cite{caro2025initialtofinalstateinverseproblemunbounded}*{Proposition 5.1}, by using the integration by parts formula of Proposition \ref{prp:integration_by_parts_unbounded_no-endpoint} in our case. 

We record the following lemma, which converts an orthogonality condition on the source term F against physical final-state solutions into a vanishing final condition for the corresponding inhomogeneous solution.
\begin{lemma}\label{lem:CR_lemma4.4_Lp}\sl Consider $V\in L^q(\R^n)$ for some $q\in(1,\infty]$ if $n=2$, or for some $q\in[n/2,\infty]$ if $n\geq 3$, let $(r, p)$ be a Strichartz pair with $\frac{1}{q}+\frac{2}{p}=1$, and $F\in L^{r'}((0,T); L^{p'}(\R^n))$ with 
\[
\int_\Sigma F \ovl{v} = 0
\]
for every $v\in C([0, T]; L^2(\R^n)) \cap L^r((0, T); L^p(\R^n))$  which is the solution of the final-value problem 
\[
\begin{cases}
i\d_t v = - \Delta v + \ovl{V} v\quad &\textrm{in}\quad \Sigma,
\vspace{.6em}
\\
v(T, \cd)= g  \quad &\textrm{in}\quad \R^n.
\end{cases}
\]
with arbitrary $g\in L^2(\R^n)$. Then, the solution $u\in C([0, T]; L^2(\R^n)) \cap L^r((0, T); L^p(\R^n))$ of the problem
\[
\begin{cases}
\left(i\partial_t +\Delta\right)u-Vu=F\quad &\textrm{in}\quad \Sigma,
\vspace{.6em}
\\
u(0,\cd)=0, \quad &\textrm{in}\quad \R^n.
\end{cases}
\]
satisfies $u(T, \cd) = 0. $
\end{lemma}

The proof is in the same spirit as~\cite{caro2025initialtofinalstateinverseproblemunbounded}*{Lemma 5.3}, but we include it also here for reader's convenience since it is rather short.

\begin{proof} Let $g\in L^2(\R^n)$ and $v$ be the corresponding solution of the final-value problem in the statement of the lemma. Under these assumptions, Proposition~\ref{prp:integration_by_parts_unbounded_no-endpoint} applies to $v$ yielding
\begin{align*}
i \int_{\R^n} u(T, \cd) \ovl{g} &= \int_\Sigma \left[ \left(i \d_t +\Delta\right)u \ovl{v} - u \ovl{(i\d_t + \Delta)v} \right]
= \int_\Sigma \left[ \left(i \d_t +\Delta - V\right)u\ovl{v} - u \ovl{\left(i \d_t + \Delta - \ovl{V}\right)v}  \right] 
\\
&= \int_\Sigma F \ovl{v}  = 0.
\end{align*}
Since this holds for every $g\in L^2(\R^n)$, we  conclude that $u(T, \cd) = 0$. 
\end{proof}

The following lemma is a symmetric version of Lemma~\ref{lem:CR_lemma4.4_Lp}; we omit the completely analogous proof.

\begin{lemma}\label{lem:CR_lemma4.4_Lp_symmetric}\sl Consider $V\in  L^q(\R^n)$ for some $q\in(1,\infty]$ if $n=2$, or for some $q\in[n/2,\infty]$ if $n\geq 3$, let $(r, p)$ be a Strichartz pair with $\frac{1}{q}+\frac{2}{p}=1$, and $G\in L^{r'}((0,T); L^{p'}(\R^n))$ such that 
\[
\int_\Sigma \ovl{G} u = 0
\]
for every $u \in C([0, T]; L^2(\R^n)) \cap L^r((0, T); L^p(\R^n)) $ which is the solution of the initial-value problem \eqref{eq:Schrodinger} for arbitrary $f\in L^2(\R^n)$. Then, the solution $v \in C([0, T]; L^2(\R^n)) \cap L^r((0, T); L^p(\R^n))$ of the problem 
\[
\begin{cases}
\left(i\d_t + \Delta - \ovl{V}\right) v = G \quad &\textrm{in}\quad \Sigma, 
\vspace{.6em}\\
v(T, \cd) = 0  \quad &\textrm{in}\quad \R^n.
\end{cases} 
\]
satisfies that $v(0, \cd)=0$.
\end{lemma}

We are now in a position to prove the main orthogonality relation of this paper, in which the physical solutions appearing in Proposition~\ref{prp:CR_prop4.1_unbounded} are replaced by the stationary state solutions constructed in Section~\ref{sec:timeindependent_solutions}.

\begin{proposition}\label{pr:int_by_parts_time-harmonic_solutions}\sl Let $V_1,V_2\in L^1(\R^n)\cap L^q(\R^n)$, where $q\in(1,(n+1)/2]$ if $n=2$, or $q\in[n/2,(n+1)/2]$ if $n\geq 3$, and define $p$ by 
\[
  \frac{1}{q}+\frac{2}{p}=1.
\]
Let $\mathcal U_T^1,\mathcal U_T^2$ denote the initial-to-final-state maps corresponding to $V_1,V_2$, respectively.

Assume that $\mathcal U_T^1=\mathcal U_T^2$. Then, for every $\lambda>0$, every $\omega_1,\omega_2\in\Sph^{n-1}$, and every pair of functions
\[
\psi_j(t,x)=e^{-i\lambda^2 t}\left(w_j^{(0)}(x)+w_j^{\mathrm{cor}}(x)\right),
\qquad j\in\{1,2\},
\]
with
\[
w_j^{(0)}(x)=e^{-i\lambda\omega_j\cdot x},  \qquad w_j^{\mathrm{cor}}\in L^p(\R^n),
\]
satisfying
\[
\left(i\partial_t+\Delta-V_1\right)\psi_1=0, \qquad \left(i\partial_t+\Delta-\overline{V_2}\right)\psi_2=0,
\]
the following orthogonality relation holds:
\begin{equation}\label{eq:int_by_parts_time-harmonic} \int_\Sigma (V_1-V_2) \psi_1 \overline{\psi_2}=0.
\end{equation}

In particular, \eqref{eq:int_by_parts_time-harmonic} holds for the stationary-state solutions constructed in Corollaries~\ref{cor:solutions} and~\ref{cor:solutionsVcrit}.
\end{proposition}

\begin{proof}  We argue by combining the orthogonality relation for physical solutions with suitable auxiliary boundary value problems in order to obtain \eqref{eq:int_by_parts_time-harmonic}. This is implemented in two steps. 

\vspace{1em}
\noindent\textbf{Step 1:} In this first step, we prove that
\[
\cal{U}^1_T=\cal{U}^2_T \Rightarrow \int_\Sigma (V_1 - V_2) \Psi \ovl{\psi_2} = 0 ,
\]
for any $\Psi \in C([0, T]; L^2(\R^n)) \cap L^r((0, T); L^p(\R^n))$ satisfying $(i \d_t + \Delta - V_1)\Psi = 0$ in $\Sigma$. 

To that end, for any such solution \(\Psi\) as above, consider the problem
\begin{equation}\label{eq:initial_BVP_F}
\begin{cases}
\left(i \d_t + \Delta - V_2\right )\Phi = (V_1 - V_2)\Psi \eqqcolon F  \quad &\textrm{in}\quad \Sigma, 
\vspace{.6em}\\
\Phi(0, \cd) = 0  \quad &\textrm{in}\quad \R^n.
\end{cases}
\end{equation}
Using that $V_1-V_2 \in L^q(\R^n)$ and \eqref{eq:holder1} we have $F\in L^{r'}((0,T);L^{p'}(\R^n))$, and hence, by the discussion in \S\ref{subsec: Stri_WP}, there exists a unique solution \(\Phi\) of \eqref{eq:initial_BVP_F} satisfying
\[
\Phi\in C\left([0,T];L^2(\R^n)\right)\cap L^r\left((0,T);L^p(\R^n)\right).
\] 
Now, since $\mathcal U_T^1=\mathcal U_T^2$, Proposition~\ref{prp:CR_prop4.1_unbounded}
implies that
\[
\int_\Sigma (V_1-V_2)\Psi \overline{v_2}=0
\]
for every physical final-state solution $v_2$ associated with $\ovl{V_2}$. In particular, this holds for the physical solution $v_2$ of the final-value problem with potential $\ovl{V_2}$ and final data $g=\Phi(T,\cd)$. Applying Lemma~\ref{lem:CR_lemma4.4_Lp} with $u=\Phi$ and $V=\ovl{V_2}$, we conclude that
$\Phi(T,\cd)=0$.

We now plan to apply Lemma~\ref{lem:int_by_parts_complex-exp_Lp} with $\Phi$ in the place of $u$. For this, we need to verify that
\[
\left(i\partial_t+\Delta\right)\Phi \in L^1(\Sigma)\cap L^{r'}\left((0,T);L^{p'}(\R^n)\right) \quad\text{and}\quad \Phi(0,\cd)=\Phi(T,\cd)=0.
\]
The boundary condition $\Phi(0,\cd)=0$ holds by construction, while we already proved that $\Phi(T,\cd)=0$. Moreover, from \eqref{eq:initial_BVP_F} we have
\[
\left(i\partial_t+\Delta\right)\Phi = V_2\Phi + F,
\]
so it suffices to verify that $V_2\Phi$ and $F$ belong to $L^1(\Sigma)\cap L^{r'}((0,T);L^{p'}(\R^n))$.

First, we check the $L^{r'}((0,T);L^{p'}(\R^n))$-bounds. Since $V_1-V_2,\,V_2\in L^q(\R^n)$ and $\Psi,\Phi\in L^{r}((0,T);L^p(\R^n))$ with $\frac1q+\frac2p=1$, we apply \eqref{eq:holder1} to get
\[
(V_1-V_2)\Psi \in L^{r'}\bigl((0,T);L^{p'}(\R^n)\bigr), \qquad V_2\Phi \in L^{r'}\bigl((0,T);L^{p'}(\R^n)\bigr).
\]

Next, we verify the $L^1(\Sigma)$-bounds. Since $V_j\in L^1(\R^n)\cap L^q(\R^n)$ and $p'\in (1,q)$, we have that $V_1-V_2,\,V_2\in L^{p'}(\R^n)$.  H\"older’s inequality in space yields
\[
\int_{\R^n} |(V_1-V_2)(x)\Psi(t,x)|\,\dd x \le \|V_1-V_2\|_{L^{p'}(\R^n)} \|\Psi(t,\cd)\|_{L^p(\R^n)},
\]
and similarly
\[
\int_{\R^n} |V_2(x)\Phi(t,x)|\,\dd x \le \|V_2\|_{L^{p'}(\R^n)} \|\Phi(t,\cd)\|_{L^p(\R^n)}.
\]
Integrating these inequalities for $t\in(0,T)$ and using that $\Psi,\Phi\in L^{r}\bigl((0,T);L^p(\R^n)\bigr)$, we conclude that
\[
F=(V_1-V_2)\Psi \in L^1(\Sigma), \qquad V_2\Phi \in L^1(\Sigma).
\]
We also record for later use that  $(V_1-V_2)\Psi \ovl{\psi_2}\in L^1(\Sigma)$. This follows from a similar calculation, using $|\psi_2|\le 1+|w_2^{\mathrm{cor}}|$ and $w_2^{\mathrm{cor}}\in L^p(\R^n)$.

Hence, Lemma~\ref{lem:int_by_parts_complex-exp_Lp} applies and yields
\begin{equation}\label{eq:IBP_Phi_exp}
\int_\Sigma e^{i(\lambda^2 t+\lambda\omega_2\cdot x)}(i\partial_t+\Delta)\Phi = 0,
\end{equation}
with $\lambda>0$ and $\omega_2\in\Sph^{n-1}$ as in the statement.

Using that $\Phi$ satisfies \eqref{eq:initial_BVP_F} with $(V_1-V_2)\Psi \ovl{\psi_2}\in L^1(\Sigma)$, we may write
\[
\int_\Sigma (V_1 - V_2)\Psi \ovl{\psi_2}= \int_\Sigma \left(i\d_t + \Delta - V_2\right)\Phi  e^{i\lambda^2 t}\left(e^{i\lambda\omega_2\cdot x}+\ovl{w^{\mathrm{cor}}_2}\right).
\]
Then, expanding the right-hand side and using \eqref{eq:IBP_Phi_exp}, we obtain
\begin{equation}\label{eq:step1}
\int_\Sigma (V_1 - V_2)\Psi \ovl{\psi_2}= - \int_\Sigma \Phi V_2  \ovl{e^{-i(\lambda^2 t+\lambda\omega_2\cdot x)}} + \int_\Sigma  \Phi \ovl{\left(i\d_t+\Delta-\ovl{V_2}\right)\left(e^{-i\lambda^2 t}w^{\mathrm{cor}}_2\right)}.
\end{equation}
However, since $\psi_2=e^{-i\lambda^2 t}(w^{(0)}_2+w^{\mathrm{cor}}_2)$ solves $(i\partial_t+\Delta-\ovl{ V_2})\psi_2=0$, and $e^{-i\lambda^2 t}w^{(0)}_2$ solves the free Schr\"odinger equation, we have
\[
\left(i\d_t+\Delta-\ovl{V_2}\right)\left(e^{-i\lambda^2 t}w^{\mathrm{cor}}_2\right) =\ovl{V_2} e^{-i(\lambda^2 t+\lambda\omega_2\cdot x)}.
\]
The two terms in the right hand side of \eqref{eq:step1} therefore cancel, and we conclude that
\[
\int_\Sigma (V_1 - V_2)\Psi \ovl{\psi_2} = 0.
\]
This completes the proof of the first step.

\vspace{1em}
\noindent\textbf{Step 2:} In this second part of the proof, we use the conclusion established in the first step, applied to solutions $\Psi$ of the Schrödinger equation with potential $V_1$, to deduce that
\[
\int_\Sigma (V_1-V_2)\psi_1\ovl{\psi_2}=0.
\]
To that end, let us consider the final-value problem
\begin{equation}\label{eq:FVP_G}
\begin{cases}
\left(i\d_t + \Delta - \ovl{V_1}\right)\Xi= \left(\ovl{V_1-V_2}\right)\psi_2 \eqqcolon G  \quad &\textrm{in}\quad \Sigma, 
\vspace{.6em}\\
\Xi(T,\cd) = 0   \quad &\textrm{in}\quad \R^n.
\end{cases}
\end{equation}
Using $\eqref{eq:holder1}$ we have $G\in L^{r'}((0,T);L^{p'}(\R^n))$, and hence by the well-posededness discussion in \S\ref{subsec: Stri_WP} there exists a unique solution
\[
\Xi\in C\left([0,T];L^2(\R^n)\right)\cap L^r\left((0,T);L^p(\R^n)\right).
\]
We now verify the assumptions needed to apply Lemma~\ref{lem:CR_lemma4.4_Lp_symmetric} to the solution $\Xi$. As before, we have $ | \psi_2 | \leq 1 + |w^{\mathrm{cor}}_2| $ with $w^{\mathrm{cor}}_2 \in L^p(\R^n)$, and from \eqref{eq:holder1} that $\ovl{G} = (V_1 - V_2)\ovl{\psi_2} \in L^{r'}((0, T); L^{p'}(\R^n))$. Thus $\ovl{G}\in L^{r'}((0,T);L^{p'}(\R^n))$, and since $\Psi\in L^r((0,T);L^p(\R^n))$
the dual pairing $\int_\Sigma \ovl{G} \Psi$ is well defined. Applying the conclusion of the first part of the proof to this $\Psi$, we obtain
\[
\int_\Sigma \ovl{G} \Psi = \int_\Sigma (V_1 - V_2)\Psi\ovl{\psi_2} = 0 .
\]
Using the identity above in Lemma~\ref{lem:CR_lemma4.4_Lp_symmetric} with this choice of $G$ and the solution $\Xi$ of \eqref{eq:FVP_G} yields $\Xi(0,\cd)=0$. Using the assumptions $V_j\in L^1(\R^n)\cap L^q(\R^n)$ and $w^{\mathrm{cor}}_j\in L^p(\R^n)$ for $j\in\{1,2\}$, we have $(V_1-V_2)\psi_1\ovl{\psi_2}\in L^1(\Sigma)$. Moreover, by \eqref{eq:FVP_G} and \eqref{eq:holder}, \eqref{eq:holder1}, 
\[
\left(i\partial_t+\Delta\right)\Xi=\ovl{V_1}\Xi+G \in L^1(\Sigma)\cap L^{r'}\big((0,T);L^{p'}(\R^n)\big).
\]
Since $\Xi$ satisfies~\eqref{eq:FVP_G}, we may write
\[
\int_\Sigma (V_1 - V_2)\psi_1 \ovl{\psi_2} = \int_\Sigma \psi_1 \ovl{\left(i\d_t + \Delta - \ovl{V_1}\right) \Xi}.
\]
Expanding $\psi_1=e^{-i(\lambda^2 t+\lambda\omega_1\cdot x)}+e^{-i\lambda^2 t}w^{\mathrm{cor}}_1$ and replacing in the right hand side of the display above, we obtain
\begin{equation}\label{eq:long}
\begin{split}
\int_\Sigma (V_1 - V_2)\psi_1 \ovl{\psi_2} &= \int_\Sigma e^{-i (\lambda^2 t + \lambda\omega_1\cdot x)} \ovl{\left(i \d_t +\Delta\right ) \Xi} - \int_\Sigma e^{-i (\lambda^2 t + \lambda \omega_1\cdot x)} V_1 \ovl{\Xi} 
\\
&\quad + \int_\Sigma e^{-i\lambda^2 t}w^{\mathrm{cor}}_1 \ovl{\left(i \d_t + \Delta\right) \Xi} - \int_\Sigma e^{-i\lambda^2 t}w^{\mathrm{cor}}_1 V_1 \ovl{\Xi}.
\end{split}
\end{equation}
Now, by Lemma~\ref{lem:int_by_parts_complex-exp_Lp} applied to $\Xi$, we have for the first summand in the right hand side of \eqref{eq:long}
\[
\int_\Sigma e^{-i (\lambda^2 t + \lambda\omega_1\cdot x)} \ovl{\left(i \d_t +\Delta\right) \Xi}=0.
\]
For the third term in the right hand side of \eqref{eq:long}, we apply Proposition~\ref{cor:integration_by_parts_unbounded_no-endopint2}
with $u=\Xi$ and $v=e^{-i\lambda^2 t}w^{\mathrm{cor}}_1$ and obtain
\[
\int_\Sigma e^{-i\lambda^2 t}w^{\mathrm{cor}}_1 \ovl{\left(i \d_t + \Delta\right)\Xi} = \int_\Sigma \ovl{\left(i\d_t+\Delta\right)\left(e^{-i\lambda^2 t}w^{\mathrm{cor}}_1\right)} \Xi.
\]
Since $\psi_1=e^{-i\lambda^2 t}(w^{(0)}_1+w^{\mathrm{cor}}_1)$ solves $(i\partial_t+\Delta-V_1)\psi_1=0$ and $(i\partial_t+\Delta) (e^{-i\lambda^2 t}w^{(0)}_1)=0$,
it follows that
\[
\left(i\d_t+\Delta\right)\left(e^{-i\lambda^2 t}w^{\mathrm{cor}}_1\right) = e^{-i\lambda^2 t}V_1\left(w^{(0)}_1+w^{\mathrm{cor}}_1\right) = V_1 e^{-i (\lambda^2 t + \lambda\omega_1\cdot x)}+ V_1 e^{-i\lambda^2 t}w^{\mathrm{cor}}_1 .
\]
Substituting this identity into the previous display, we get
\[
\int_\Sigma e^{-i\lambda^2 t}w^{\mathrm{cor}}_1 \ovl{\left(i \d_t + \Delta\right)\Xi} = \int_\Sigma V_1 e^{-i \left(\lambda^2 t + \lambda\omega_1\cdot x\right)} \ovl{\Xi} + \int_\Sigma V_1 e^{-i\lambda^2 t}w^{\mathrm{cor}}_1 \ovl{\Xi}.
\]
Therefore the remaining three terms in \eqref{eq:long} cancel, and we conclude that
\[
\int_\Sigma (V_1 - V_2)\psi_1 \ovl{\psi_2}=0,
\]
which proves~\eqref{eq:int_by_parts_time-harmonic}.
\end{proof}

\section{The proof of Theorem~\ref{th:Lq}}\label{sec:unbounded_endpoint} We split the proof of Theorem~\ref{th:Lq} into two subsections. The first contains the proof of the non-endpoint case, namely the case that $V\in L^q(\R^n)$ in dimensions $n\geq 2$, for some $q>n/2$. In the subsequent section we will provide the proof of the endpoint case $V\in L^{n/2}(\R^n)$ for dimension $n\geq 3$.

\subsection{Proof of Theorem~\ref{th:Lq} in the non-endpoint case}\label{subsec:proof_uniqueness_unbounded} Here we consider potentials $V_1,V_2\in L^q(\R^n)$ for some $q>n/2$ and $n\geq 2$. By the discussion in Remark~\ref{rmrk:numer}, we may and do assume that $n/2<q\leq (n+1)/2$, and define $p$ through the relation $\frac{1}{q}=1-\frac{2}{p}$.

 Let $F\coloneqq V_1 - V_2 \in L^1(\R^n)\cap L^q(\R^n)$.  Since $1<p'<q$, we also have that $F\in L^{p'}(\R^n)$. Given $\xi \in \R^n$ consider $\nu \in \Sph^{n-1}$ such that $\xi \cdot \nu = 0$.
For $\lambda \geq |\xi|/2$ we define
\[
\begin{split}
\omega_1 \coloneqq \frac{1}{\lambda} \frac{\xi}{2} + \left( 1 - \frac{|\xi|^2}{4 \lambda^2} \right)^{1/2} \nu,\qquad \omega_2 \coloneqq -\frac{1}{\lambda} \frac{\xi}{2} +\left( 1 - \frac{|\xi|^2}{4 \lambda^2} \right)^{1/2} \nu.
\end{split}
\]
Note that $\omega_1$ and $\omega_2$ belong to $\Sph^{n-1}$ with $\omega_1 - \omega_2 = \xi/\lambda$. We construct the solutions $\psi_1,\psi_2$ solving $(i\partial_t+\Delta-V_1)\psi_1=0$, and $(i\partial_t+\Delta-\ovl{V_2})\psi_2=0$ of the form
\[
\psi_j\coloneqq \psi_{V_j} ^{\lambda,\omega_j}=e^{-i\lambda^2t}w_j(x)=e^{-i\lambda^2t}(w^{(0)}_j(x)+w^{\mathrm{cor}}_j(x)),\qquad j\in\{1,2\},\qquad (t,x)\in\R\times \R^n,
\]
where we remember that $w^{(0)}_j(x)=e^{-i\lambda \omega_j\cdot x}$ and the functions $w^{\mathrm{cor}}_j$ are constructed in Corollary~\ref{cor:solutions} with $\lambda > \max(|\xi|/2, \lambda_{V_j})$. 

In particular, the corrections $w^{\mathrm{cor}}_j$ satisfy the following $L^p$-estimates
\begin{equation}\label{eq:decay_remainder_Lp}
\left\| w^{\mathrm{cor}}_j \right\|_{L^p(\R^n)} \lesssim  \frac{1}{\lambda^{n(\frac{2}{n} - \frac{1}{q})}} \| V_j \|_{L^{p'} (\R^n)},\qquad j\in\{1,2\},
\end{equation}
for all $\lambda \geq \max(|\xi|/2, \lambda_{V_j} + \delta)$. Since $ \mathcal{U}^1_T=\mathcal{U}^2_T$, plugging the solutions $\psi_1,\psi_2$ into the orthogonality relation of Proposition~\ref{pr:int_by_parts_time-harmonic_solutions}, and noting that the time-dependent phases cancel, we infer that
\begin{equation}\label{eq:leading=remainder_unbounded}
\int_{\R^n} F w^{(0)}_1 \ovl{w^{(0)}_2}  = - \int_{\R^n} F \left[w^{(0)}_1 \ovl{w^{\mathrm{cor}}_2} + w^{\mathrm{cor}}_1 \ovl{w^{(0)}_2} + w^{\mathrm{cor}}_1 \ovl{w^{\mathrm{cor}}_2}\right].
\end{equation}
For the left-hand side of \eqref{eq:leading=remainder_unbounded}, note that  
\[
\int_{\R^n} F w^{(0)}_1 \ovl{w^{(0)}_2}  = \int_{\R^n} F(x) e^{-i \lambda( \omega_1-\omega_2)\cdot x}\,\dd x  = \int_{\R^n} F(x) e^{-i \xi \cdot x} \,\dd x = (2 \pi)^{n/2} \wh{F}(\xi).
\]
Using H\"older's inequality along with the estimates \eqref{eq:decay_remainder_Lp}, the right-hand side of \eqref{eq:leading=remainder_unbounded} is readily estimated as follows
\[
\begin{split}
|\wh{F}(\xi)| &\lesssim  \| F \|_{L^{p'}(\R^n)} \left\|w^{\mathrm{cor}}_2\right\|_{L^p(\R^n)} + \left\| F \right\|_{L^{p'}(\R^n)} \left\|w^{\mathrm{cor}}_1\right\|_{L^p(\R^n)} 
\\
&\qquad + \| F \|_{L^q(\R^n)} \left\| w^{\mathrm{cor}}_1 \right\|_{L^p(\R^n)} \left\| w^{\mathrm{cor}}_2 \right\|_{L^p(\R^n)} 
\\
&\lec \frac{1}{\lambda^{n(\frac{2}{n} - \frac{1}{q})}} \| F \|_{L^{p'}(\R^n)} \left(\| V_1 \|_{L^{p'}(\R^n)} +  \| V_2 \|_{L^{p'}(\R^n)} \right) 
\\
&\qquad + \frac{1}{\lambda^{2n(\frac{2}{n} - \frac{1}{q})}} \| F \|_{L^{q}(\R^n)}  \| V_1 \|_{L^{p'}(\R^n)} \| V_2 \|_{L^{p'}(\R^n)} \\
&\lec  \frac{1}{\lambda^{n(\frac{2}{n} - \frac{1}{q})}} + \frac{1}{\lambda^{2n(\frac{2}{n} - \frac{1}{q})}} .
\end{split}
\]
Letting $\lambda\to \infty $ we obtain that $\wh{F}(\xi) = 0$ since $q>n/2$. Since $\xi\in\R^n$ was arbitrary this implies that $V_1=V_2$ almost everywhere in $\R^n$, thus concluding the proof of Theorem~\ref{th:Lq} in the case $q>n/2$ and any dimension $n\geq 2$.

\subsection{Proof of Theorem~\ref{th:Lq}: the endpoint case}\label{subsec: proof_uniqueness_endpoint} Let $V_1,V_2\in L^1(\R^n)\cap L^{n/2}(\R^n)$ and define $F\coloneqq V_1 - V_2$. Given $\xi \in \R^n$ we consider $\nu \in \Sph^{n-1}$ such that $\xi \cdot \nu = 0$ and $\lambda \geq \max\{\lambda_{V_1},\lambda_{V_2},|\xi|/2\}$, where $\lambda_{V_1},\lambda_{V_2}$ are the constants corresponding to $V_1,V_2$ from Corollary~\ref{cor:solutionsVcrit}. As in \S\ref{subsec:proof_uniqueness_unbounded}, we consider the vectors
\[
\omega_1 \coloneqq \frac{1}{\lambda} \frac{\xi}{2} + \left( 1 - \frac{|\xi|^2}{4 \lambda^2} \right)^{1/2} \nu,\qquad \omega_2 \coloneqq -\frac{1}{\lambda} \frac{\xi}{2} + \left( 1 - \frac{|\xi|^2}{4 \lambda^2} \right)^{1/2} \nu,
\]
and note that $\omega_1,\omega_2\in\mathbb S^{n-1}$ and $\omega_1 - \omega_2 = \xi/\lambda$.  
We now construct the stationary state solutions as in Corollary~\ref{cor:solutionsVcrit}
\[
\psi_j\coloneqq \psi_{V_j} ^{\lambda,\omega_j}=e^{-i\lambda^2t}w_j(x)=e^{-i\lambda^2t}\left(w^{(0)}_j(x)+w^{\mathrm{cor}}_j(x)\right),\qquad j\in\{1,2\},\qquad (t,x)\in\R\times \R^n,
\]
and we recall the estimates
\begin{equation}\label{eq:psiestimates}
\left\|w^{\mathrm{cor}}_j\right\|_{X_\lambda ^*}\lesssim \|V_j\|_{X_\lambda}\leq \lambda^{-\frac{1}{n+1}} \|V_j\|_{L^1(\R^n)\cap L^{n/2}(\R^n)},\qquad j\in\{1,2\},
\end{equation}
also proved in Corollary~\ref{cor:solutionsVcrit}. Since $ \mathcal{U}^1_T=\mathcal{U}^2_T$, plugging the solutions $\psi_1,\psi_2$ into the orthogonality relation of 
Proposition~\ref{pr:int_by_parts_time-harmonic_solutions} and repeating the calculations from \S\ref{subsec:proof_uniqueness_unbounded}, we have 
\[
\begin{split}
|\wh{F}(\xi)| &\leq \frac{1}{(2\pi)^{n/2}} \int_{\R^n} \left|F w^{(0)}_1 \overline{w^{\mathrm{cor}}_2}\right| + \int_{\R^n} \left|F \overline{w^{(0)}_2} w^{\mathrm{cor}}_1\right| + \int_{\R^n} \left|F w^{\mathrm{cor}}_1 \overline{w^{\mathrm{cor}}_2}\right| 
\\
&\lesssim \left\|Fw^{(0)}_1\right\|_{X_\lambda} \left\|w^{\mathrm{cor}}_2\right\|_{\X_\lambda ^*}+\left\|Fw^{(0)}_2\right\|_{X_\lambda} \left\|w^{\mathrm{cor}}_1\right\|_{\X_\lambda ^*}+\|F\|_\lambda\|w^{\mathrm{cor}}_1\|_{\X_\lambda ^*} \left \|w^{\mathrm{cor}}_2\right \|_{\X_\lambda ^*},
\end{split}
\]
where we have used the duality of $X_\lambda$ and $X_\lambda ^*$ and the H\"older-type inequality of Lemma~\ref{lem:multiplication_V_endpoint_X_norm}. Using the estimates \eqref{eq:psiestimates} and the fact that $|w^{(0)}_1|=|w^{(0)}_2|=1$ we get
\begin{align*}
|\wh{F}(\xi)| &\lesssim 
 \|F\|_{\X_\lambda} \lambda^{-\frac{1}{n+1}}(\|V_1\|_{L^1(\R^n)\cap L^{n/2}(\R^n)} + \|V_2\|_{L^1(\R^n)\cap L^{n/2}(\R^n)}) \\
&+ \|F\|_{\lambda} \lambda^{-\frac{2}{n+1}} \|V_1\|_{L^1(\R^n)\cap L^{n/2}(\R^n)} \|V_2\|_{L^1(\R^n)\cap L^{n/2}(\R^n)} \\
&\lesssim\lambda^{-\frac{2}{n+1}} \left(\|V_1\|_{L^1(\R^n)\cap L^{n/2}(\R^n)} +\|V_2\|_{L^1(\R^n)\cap L^{n/2}(\R^n)}\right) ^2 
(1 + \|F\|_{\lambda}) \to 0
\end{align*}
as $\lambda \to \infty$. Hence $\wh{F}(\xi) = 0$, and since $\xi\in\R^n$ was arbitrary we conclude that $F=0$ and hence that $V_1=V_2$ almost everywhere in $\R^n$, thus completing the proof of the endpoint case of Theorem~\ref{th:Lq}.

\sloppy
\begin{acknowledgements}
M. Ca\~nizares is partially supported by grant PID2024-156267NB-I00 funded by MICIU/AEI/10.13039/501100011033 and cofunded by the European Union. P. Caro is supported by grant PID2024-156267NB-I00 funded by MICIU/AEI/10.13039/501100011033 and cofunded by the European Union, as well as BCAM-BERC 2022-2025 and the BCAM Severo Ochoa CEX2021-001142-S. T. Zacharopoulos is supported by the grant 10.46540/3120-00003B from Independent Research Fund Denmark. I. Parissis is partially supported by grant PID2024-156267NB-I00 funded by MICIU/AEI/10.13039/501100011033 and cofunded by the European Union, grant IT1615-22 of the Basque Government and IKERBASQUE.
\end{acknowledgements}

\bibliography{references}{} 
\end{document}